\documentclass[10pt]{article}
\usepackage{amsmath}
\usepackage{latexsym}
\usepackage{amssymb}

\usepackage{amsthm}

\font\tengoth=eufm10 at 10pt
\font\sevengoth=eufm7 at 6pt
\newfam\gothfam
\textfont\gothfam=\tengoth
\scriptfont\gothfam=\sevengoth

\newcommand{\mlabel}[1]{\label{#1}}

\newcommand{\g}{{\mathfrak g}}

\newcommand{\h}{{\mathfrak h}}

\newcommand{\fa}{{\mathfrak a}}

\newcommand{\fg}{{\mathfrak g}}
\newcommand{\fh}{{\mathfrak h}}

\newcommand{\fo}{{\mathfrak o}}
\newcommand{\fq}{{\mathfrak q}}

\newcommand{\fu}{{\mathfrak u}}

\newcommand{\bx}{\mathbf{x}}

\renewcommand{\:}{\colon}
\newcommand{\1}{\mathbf{1}}

\newcommand{\cA}{\mathcal{A}}

\newcommand{\cD}{\mathcal{D}}
\newcommand{\cE}{\mathcal{E}}
\newcommand{\cF}{\mathcal{F}}

\newcommand{\cH}{\mathcal{H}}

\newcommand{\cL}{\mathcal{L}}

\newcommand{\cN}{\mathcal{N}}
\newcommand{\cO}{\mathcal{O}}

\newcommand{\cV}{\mathcal{V}}

\newcommand{\eset}{\emptyset}

\newcommand{\derat}[1]{\frac{d}{dt} \hbox{\vrule width0.5pt
                height 5mm depth 3mm${{}\atop{{}\atop{\scriptstyle t=#1}}}$}}

\renewcommand{\phi}{\varphi}
\newcommand{\dd}{{\tt d}}

\newcommand{\subeq}{\subseteq}
\newcommand{\supeq}{\supseteq}

\newcommand{\eps}{\varepsilon}

\newcommand{\N}{{\mathbb N}}

\newcommand{\R}{{\mathbb R}}
\newcommand{\C}{{\mathbb C}}

\renewcommand{\hat}{\widehat}

\renewcommand{\tilde}{\widetilde}

\newcommand{\GL}{\mathop{{\rm GL}}\nolimits}

\newcommand{\SO}{\mathop{{\rm SO}}\nolimits}

\newcommand{\OO}{\mathop{\rm O{}}\nolimits}

\newcommand{\U}{\mathop{\rm U{}}\nolimits}

\newcommand{\ad}{\mathop{{\rm ad}}\nolimits}
\newcommand{\Ad}{\mathop{{\rm Ad}}\nolimits}

\renewcommand{\Re}{\mathop{{\rm Re}}\nolimits}

\newcommand{\End}{\mathop{{\rm End}}\nolimits}
\newcommand{\id}{\mathop{{\rm id}}\nolimits}

\renewcommand{\dim}{\mathop{{\rm dim}}\nolimits}

\newcommand{\supp}{\mathop{{\rm supp}}\nolimits}

\newcommand{\Spann}{\mathop{{\rm span}}\nolimits}
\newcommand{\ev}{\mathop{{\rm ev}}\nolimits}

\newcommand{\oline}{\overline}

\newcommand{\la}{\langle}
\newcommand{\ra}{\rangle}

\newcommand{\Mot}{{\rm Mot}}

\newcommand{\res}{\vert}

\newcommand{\Spin}{{\rm Spin}}

\newcommand{\ssssarr}{\hbox to 15pt{\rightarrowfill}}
\newcommand{\sssarr}{\hbox to 20pt{\rightarrowfill}}
\newcommand{\ssarr}{\hbox to 30pt{\rightarrowfill}}
\newcommand{\sarr}{\hbox to 40pt{\rightarrowfill}}
\newcommand{\arr}{\hbox to 60pt{\rightarrowfill}}
\newcommand{\larr}{\hbox to 60pt{\leftarrowfill}}
\newcommand{\Arr}{\hbox to 80pt{\rightarrowfill}}

\def\theoremname{Theorem}
\def\propositionname{Proposition}
\def\corollaryname{Corollary}
\def\lemmaname{Lemma}
\def\remarkname{Remark}
\def\conjecturename{Conjecture} 

\def\definitionname{Definition}
\def\exercisename{Exercise}
\def\examplename{Example}
\def\examplesname{Examples}
\def\problemname{Problem}
\def\problemsname{Problems}
\def\proofname{Proof}

\def\satzname{Satz} 
\def\koroname{Korollar}
\def\folgname{Folgerung}
\def\bemerkname{Bemerkung}
\def\aufgname{Aufgabe}

\def\beisname{Beispiel}
\def\beissname{Beispiele}
\def\bewname{Beweis}

\def\@thmcounter#1{\noexpand\arabic{#1}}
\def\@thmcountersep{}
\def\@begintheorem#1#2{\it \trivlist \item[\hskip 
\labelsep{\bf #1\ #2.\quad}]}
\def\@opargbegintheorem#1#2#3{\it \trivlist
      \item[\hskip \labelsep{\bf #1\ #2.\quad{\rm #3}}]}
\makeatother
\newtheorem{theor}{\theoremname}[section]
\newtheorem{propo}[theor]{\propositionname}
\newtheorem{coro}[theor]{\corollaryname}
\newtheorem{lemm}[theor]{\lemmaname}

\newenvironment{thm}{\begin{theor}\it}{\end{theor}}
\newenvironment{theorem}{\begin{theor}\it}{\end{theor}}

\newenvironment{prop}{\begin{propo}\it}{\end{propo}}

\newenvironment{cor}{\begin{coro}\it}{\end{coro}}

\newenvironment{lem}{\begin{lemm}\it}{\end{lemm}}

\newtheorem{rema}[theor]{\remarkname}

\newenvironment{rem}{\begin{rema}\rm}{\end{rema}}

\newtheorem{stepnow}[theor]{}

\newtheorem{defin}[theor]{\definitionname} 

\newenvironment{defn}{\begin{defin}\rm}{\end{defin}}

\newtheorem{exerc}{\exercisename}[section]

\newtheorem{exa}[theor]{\examplename}

\newenvironment{ex}{\begin{exa}\rm}{\end{exa}}

\newtheorem{exas}[theor]{\examplesname}

\newenvironment{exs}{\begin{exas}\rm}{\end{exas}}

\newtheorem{conj}[theor]{\conjecturename}

\newtheorem{pro}[theor]{\problemname}

\newtheorem{prs}[theor]{\problemsname}

\newenvironment{Proof*}{\begin{trivlist}\item[\hskip%
\labelsep{\bf\proofname.\quad}]}%
{\end{trivlist}}

\newenvironment{prf}{\begin{proof}}{\end{proof}}
  
{\hfill\qed\end{trivlist}}

\newenvironment{beweis*}{\begin{trivlist}\item[\hskip%
\labelsep{\bf\bewname.\quad}]}%
{\end{trivlist}}

\newtheorem{satzn}[theor]{\satzname}

\newtheorem{koro}[theor]{\koroname}

\newtheorem{folg}[theor]{\folgname}

\newtheorem{bem}[theor]{\bemerkname}

\newtheorem{aufg}[theor]{\aufgname}

\newtheorem{aufgn}[theor]{\aufgname}

\newtheorem{beis}[theor]{\beisname}

\newtheorem{beiss}[theor]{\beissname}

\def\date{6.6.2014} 
\usepackage{hyperref}
\usepackage{amssymb, amsmath, enumerate}

\renewcommand{\phi}{\varphi} 

\newcommand{\di}{C^{-\infty}_c}

\def\sideremark#1{\ifvmode\leavevmode\fi\vadjust{\vbox to0pt{\vss
 \hbox to 0pt{\hskip\hsize\hskip1em
\vbox{\hsize2cm\tiny\raggedright\pretolerance10000 
 \noindent #1\hfill}\hss}\vbox to8pt{\vfil}\vss}}} 

\newcommand{\edz}[1]{\sideremark{#1}}

\numberwithin{equation}{section}
\begin{document} 


\title{Integrability of unitary representations \\ 
on reproducing kernel spaces} 
\author{ St\'ephane Merigon, Karl-Hermann Neeb, Gestur \'Olafsson}

\maketitle


\begin{abstract} Let $\fg$ be a Banach Lie algebra and $\tau : \fg \to \fg$ an involution. 
Write 
$\fg=\fh\oplus \fq$ for the eigenspace decomposition of $\fg$ with respect to $\tau$ 
and $\g^c := \fh\oplus i\fq$ for the dual Lie algebra. 
In this article we show the integrability 
of two types of infinitesimally unitary representations of $\fg^c$.
The first class of representation is determined by a smooth positive definite kernel 
$K$ on a locally convex manifold $M$. The kernel is assumed to satisfying a natural invariance condition 
with respect to an infinitesimal action 
$\beta \: \g \to \cV(M)$ by locally integrable vector fields that is compatible with a 
smooth action of a connected Lie group $H$ with Lie algebra $\fh$.
The second class is constructed from a positive definite kernel 
corresponding to a positive definite distribution 
$K \in C^{-\infty}(M \times M)$ on a finite dimensional smooth manifold $M$ 
which satisfies a similar invariance condition with respect to 
a homomorphism $\beta \: \g \to \cV(M)$. 
As a consequence, we get 
a generalization of the L\"uscher--Mack Theorem
which applies to a class of semigroups that need not have a 
polar decomposition. Our integrability results 
also apply naturally to local representations and representations 
arising in the context of reflection positivity. 
%
\end{abstract}

\section{Introduction} 

Let $G$ be a Banach--Lie group, let $\tau$ be an involutive automorphism of $G$ 
and let $H$ be an open subgroup of $G^\tau$. We call the triple 
$(G,H,\tau)$, respectively the pair $(G,\tau)$, a {\it symmetric Lie group}. 
We use the same notation for the involution induced on the Lie algebra 
$\g$ of $G$ and call $(\g,\fh,\tau)$, respectively $(\g,\tau)$, a {\it symmetric 
Lie algebra}. Write $\g = \fh \oplus \fq = \ker(\tau-\1) \oplus \ker(\tau+\1)$ 
for the eigenspace decomposition of $\g$ with respect to $\tau$. 
We have  $[\fh ,\fh]\subeq \fh, [\fq,\fq]\subseteq \fh$ and
$[\fh,\fq]\subset \fq$. Thus
$\g^c := \fh\oplus i\fq$ is a Lie algebra and 
$\tau^c(x+iy) := x-iy$ leads to the symmetric Lie algebra $(\g^c,\tau^c)$, 
called the {\it $c$-dual} of $(\fg,\tau)$. Note that the $c$-dual 
of $(\g^c,\tau^c)$ is the original symmetric Lie algebra  $(\g,\tau)$. 
We denote by $G^c$ the simply connected Lie group with Lie algebra $\fg^c$.

When $\fg$ is a semisimple Lie algebra and
$\tau$ is a Cartan involution, then $c$-duality corresponds to the well-known duality 
between Riemannian symmetric spaces of non-compact type and compact 
type already studied by \'E.~Cartan almost a century ago. 

In this paper we address  the following 
{\it  integrability problem:} 
Suppose we are given a unitary representation $(\pi_H, \cH)$ of $H$, 
a subspace $\cD \subeq \cH$, and 
a representation $\beta_c \: \g^c \to \End(\cD)$ 
by skew-symmetric operators on $\cD$ such that 
$\beta_c(x) = \dd\pi_H(x)\res_{\cD}$ for $x \in \fh$. 
When is there a unitary representation $\pi^c$ of $G^c$ on $\cH$ 
with $\cD \subeq \cH^\infty$ and $\dd\pi^c(x) \res_\cD = \beta_c(x)$ for $x \in \g^c$? 
In this sense we are asking for integrability criteria for 
a representation of the pair $(\g^c,H)$ to a representation $\pi^c$ of the 
group $G^c$.  If the operators $i \beta_c(x)$, $x \in \fq$, are essentially selfadjoint, 
then $\pi^c$ is uniquely determined. 

Note that, for the one-dimensional Lie algebra $\g = \R$, $\tau = - \id$, and 
$H = \{\1\}$), 
the problem is equivalent to finding selfadjoint extensions of a symmetric operator. In 
the general case a crucial step is the passage from $(\g^c,H)$ to $G^c$ 
is the existence of selfadjoint extension of the operators 
$i\beta_c(y)$ for $y \in i\fq$, so that Stone's Theorem leads to corresponding 
unitary one-parameter groups. To achieve this step we shall 
use Fr\"ohlich's criterion for essential selfadjointness of a symmetric operator 
$A$ which derives this proprty from the existence of sufficiently many solutions 
of the ODE $\dot\gamma(t) = A\gamma(t)$ (\cite{Fro80}). In our context 
this is much more natural than to use Nelson's Criterion 
(\cite[Lemma~5.1]{Nel59}) which is based on the existence of analytic vectors. 
One can actually translate between these two perspectives because 
any analytic vector provides a convergent power series solving the ODE and, 
conversely, the solutions of the ODE are actually analytic on open 
intervals (cf.\ \cite{Sh84}). 

Our main motivation to study this kind of integrability problem 
comes from our recent work on reflection positive unitary representations 
of $(G,\tau$) by the last two authors (cf.\ \cite{NO13, NO14}). 
Originally, reflection positivity (also called Osterwalder--Schrader positivity) 
is part of a duality between euclidean quantum field 
theories and relativistic quantum field theories (\cite{OS73}). It can be realized
by analytic  continuation in the time variable from 
the real  to the imaginary axis. 
A central difficulty in this approach is to show that the infinitesimal  relativistic
system obtained form the euclidean one can be integrated to
the relativistic symmetry group, that is, that
the corresponding representation of the Lie algebra  of the
Poincar\'e group can be integrated to a unitary representation of the group. 
The duality between the Poincar\'e group and the
euclidean motion group is a special case of the $c$-duality introduced above. 

Another very interesting occurrence
of $c$-duality interesting for representation theory is obtained as follows.
Assume that $(G,H,\tau)$ is a simple symmetric Lie group and that $H$ is connected.  
Let $\theta$ be a Cartan involution commuting with $\tau$. 
We further assume that the center 
of $G^{\theta\tau}$ is one dimensional. Then $G^c$ is a hermitian
Lie group,  $\theta^c=\theta\tau$ is a Cartan involution on $G^c$ and 
$D:=G^c/K^c$, $K^c=(G^c)^{\theta^c}$, is a bounded symmetric domain.
Furthermore, $\tau$ defines an anti-holomorphic involution on
$D$ with fixed point set $D^\tau \cong H/(H\cap K^c)$.  The
compact dual $K/(K\cap H)$ of $D_\R$ turns out to
be an $R$-symmetric space (and all $R$-symmetric spaces
are obtained this way). Its conformal group is locally
isomorphic to $G$. 
In this case $c$-duality should lead, on the level of representations,
to a correspondence between the maximally degenerate principal
series of $G$ and the unitary highest weight representations 
of $G^c$. This has been discussed in \cite{JOl98,JOl00} for the
case where $D^c$ is a tube type domain and in
\cite{NO14,Sch86} for related special cases. 
When one considers infinite dimensional versions 
of the preceding groups, $c$-duality also plays 
a crucial role in the study of the unitary representations
of $H$ (\cite{Ol84, Ol90, Ne13, MN14}). 

In the literature one finds essentially three types of representations 
that lead by analytic continuation to a unitary representation 
of the simply connected Lie group $G^c$
with Lie algebra $\fg^c$: 
\begin{itemize}
\item[\rm(L)] local representations of $G$ for which the adjoint operation 
corresponds to $g \mapsto g^\sharp := \tau(g)^{-1}$ and 
the representation is global on $H$ (cf.~\cite{Jo86, Jo87}), 
\item[\rm(RP)] reflection positive representations of $G$ (\cite{JOl98,JOl00, NO13}), 
\item[\rm(LM)] involutive representation of a subsemigroup $S\subeq G$ 
with polar decomposition $S=H\exp W$, where $W\not=\eset$ is 
an $\Ad(H)$-invariant open convex cone in $\fq$ 
(the L\"uscher--Mack Theorem \cite{LM75, MN12}).
\end{itemize}
Each of these types leads more or less 
easily to a representation $(\beta_c, \pi_H)$ of $(\g^c,H)$ as above.

In this paper we develop a uniform approach to the 
integrability of such pairs $(\beta_c,\pi_H)$ that is based on reproducing kernel
techniques. It leads to much simpler proofs, 
applies to Banach--Lie groups, and for type (LM) gives rise
to a version of the L\"uscher--Mack Theorem that applies to
semigroups without a polar decomposition and even to representations 
$(\beta_c, \pi_H)$ realized in reproducing kernel spaces on open $H$-right invariant 
domains in $G$ which are not even semigroups, and this is crucial for the 
applications to type (RP) (cf.~Theorem \ref{thm:4.11}  and Example \ref{ex:3.15}). 
For type (L) it also allows us to avoid some technical 
conditions that were used in the definition of local representations in 
\cite{Jo86,Jo87}.

The key idea is to realize the representations $(\beta_c, \pi_H)$ in a 
geometric setting which is rich enough to imply integrability to a representation
of $G^c$. This is achieved by considering Hilbert spaces $\cH$ defined 
by a smooth reproducing kernel $K$ on a locally convex manifold 
$M$\begin{footnote}{This means a smooth manifold modeled on a locally convex 
space; see \cite{Ne06} for details.}  
\end{footnote}
and which are compatible with a {\it smooth action} $(\beta,\sigma)$ of 
$(\g,H)$, which means that 
$\sigma \: M \times H \to M$ is a smooth right action 
and $\beta \: \g \to \cV(M)$ a homomorphism of Lie algebras for which the map 
$\hat\beta \: \g \times M \to TM, (x,m) \mapsto \beta(x)(m)$ is smooth,  
$\dot\sigma(x) = \beta(x)$ for $x \in \fh$, and 
each vector field $\beta(x)$, $x \in \fq$, is locally integrable, i.e., 
generates a local flow. 
In this context the compatibility between $K$ and $(\beta,\sigma)$ 
can be expressed by 
\begin{equation}\label{E:kerssym}
\cL_{\beta(x)}^1 K = -\cL_{\beta(\tau x)}^2 K\quad\text{for}\quad x\in\g
\end{equation}
Here  $\cL^1$ and $\cL^2$ denote the Lie derivative with respect to
the first and second variable, respectively.
Using a geometric version of Fr\"ohlich's Theorem on selfadjoint 
operators and corresponding local linear semiflows \cite{Fro80}, see Theorems \ref{thm:2.4} and
\ref{T:Froelichgeo} in this article, we show that
this action leads to a representation of $\fg_\C$ on a dense domain $\cD$ in the corresponding 
reproducing kernel Hilbert space $\cH_K \subeq C^\infty(M)$ 
for which $\fg^c$ acts by essentially skew-adjoint operator. The proof that
this representation integrates to a unitary representation of $G^c$
then relies on the results in \cite{Mer11}.

If $M$ is a finite dimensional manifold, a 
similar but much more general situation arises when $K$ is a positive definite 
distribution on $M \times M$. It is remarkable that in this case one can 
prove integrability along the same lines as for the case when $K$ is a smooth kernel. 
Here typical examples are obtained from 
so-called reflection positive representations of $(G,S,\tau)$, where 
$S \subeq G$ is a $\sharp$-invariant open subsemigroup. 
If $\nu$ is a 
reflection positive distribution vector, then 
\[K_\nu (f,g)=\langle \pi^{-\infty}(g * f^*)\nu, \nu\rangle 
\quad \mbox{ for } \quad f, g \in C^\infty_c(G) \]
defines a positive definite distribution on $S \times S$ to which our results apply 
(cf.\ \cite{NO14}). 

The article is organized as follows. To keep our setting as flexible as possible, we 
introduce in Section~\ref{sec:1} the concept of a 
locally integrable vector field on a locally convex manifold. 
By definition, a vector field is locally integrable if it is the velocity 
field of a local flow and the main point of Section~\ref{sec:1} is to show 
that any such vector field determines a unique maximal local flow. 
For Banach manifolds, every smooth vector field is locally integrable, but this 
is not the case for Fr\'echet spaces. 

In Section~\ref{sec:2}, Lie derivatives
of locally integrable vector fields are studied in more detail. The most important
result here is Lemma~\ref{lem:vectrafo}, asserting the compatibility of the adjoint 
representation of $\g$ and the transformation of the corresponding vector fields under flows. 
It will be used in a crucial way in the proof of the integrability theorems.

In Section~\ref{sec:2b} we use 
Fr\"ohlich's Theorem to show that a vector field that is symmetric with
respect to a smooth positive definite kernel $K$ on $M$ gives 
rise, by  its Lie derivative, to a selfadjoint operator on $\cH_K$. 

Section~\ref{sec:3} is the central section of this article. 
Here we prove  the integrability of the representation 
$(\beta_c, \pi_H)$ on the reproducing kernel space $\cH_K$ if 
\eqref{E:kerssym} is satisfied (Theorem~\ref{thm:4.8}). 
As a rather easy consequence, we obtain the following 
generalization of the L\"uscher--Mack Theorem: If 
$S = SH\subeq G$ is a $\sharp$-invariant subsemigroup with non-empty interior, then 
any smooth $*$-representation of $S$ ``extends analytically'' to a uniquely 
determined unitary representation of $G^c$. 

In Section~\ref{sec:4a} 
we first recall the definition of a local representation 
of a symmetric Lie group $(G,H,\tau)$ from \cite{Jo86, Jo87} 
and show in two different situations how it leads to a unitary representation
of $G^c$. In the fist case, where $G$ is finite dimensional, we 
give a simplified proof of Jorgensen's integrability to $G^c$. 
In the second case $G$ can be a Banach--Lie group, we assume 
that the local representation is defined on an open $\1$-neighborhood 
$U$ with $U = UH$, and that the corresponding action map 
$U \times \cD \to \cH$ is smooth for some topology on $\cD$ for which the 
inclusion $\cD \to \cH$ is continuous. 
Then integrability can be derived from Theorem~\ref{thm:4.8}. 

In Section~\ref{sec:4} we assume that the manifold $M$ is finite dimensional, 
so that the concept of a distribution on $M$ is defined. Here our main 
result is a generalization of Theorem~\ref{thm:4.8} to the situation 
where the kernel $K$ is a distribution on $M \times M$ 
satisfying \eqref{E:kerssym} in a suitable sense 
(Theorem~\ref{thm:4.12}). Again, the key tool is a suitable 
version of Fr\"ohlich's Theorem in this context. 
This results applies immediately to 
reflection positive distributions, for which it leads to an important 
class of reflection positive representations of $(G,\tau)$, for which a 
corresponding unitary representation $\pi^c$ of $G^c$ exists.

\tableofcontents

\section{Local flows on locally convex manifolds} 
\mlabel{sec:1}

A smooth manifold $M$ modeled on a locally convex space 
is called a {\it locally convex manifold}. Throughout this section, all
manifolds are assumed to be locally convex. We denote by 
$\cV(M)$ the Lie algebra of smooth vector fields on $M$ 
(cf.\ \cite{Ne06}).  

In this subsection, we show that the well-known correspondence between 
smooth local flows and vector fields can be generalized to locally convex manifolds, 
the main difference being that not every vector field generates a local flow.

For a subset $\cD\subset \R\times M$, $x\in M$ and
$t\in\R$ let
\[I_m:= \{ t \in \R \: (t,m) \in \cD\} \quad \mbox{ and } \quad 
M_t := \{ m \in M \: (t,m) \in \cD\}.\]

\begin{defn} \mlabel{def:8.6.2.1}
Let $M$ be a smooth manifold. A {\it local flow on
$M$} is a smooth map
$\Phi\colon \cD \to M, (t,x) \mapsto \Phi_t(x)$,  
where $\cD \subeq \R \times M$ is an open subset containing
$\{0\} \times M$, such that
for each $x \in M$ the set $I_x$  is an interval containing $0$, and
\begin{equation}
  \label{eq:semprop}
\Phi_0 = \id_M \quad\hbox{ and } \quad \Phi_t \Phi_s(x) = \Phi_{s+t}(x) 
\end{equation}
hold for all $t,s,x$ for which both sides are defined. The maps
\[ \gamma_x\colon I_x \to M, \qquad t \mapsto \Phi_t(x) \] 
are called the {\it flow lines}. The flow $\Phi$ is
said to be {\it global} if $\cD = \R \times M$.
\end{defn}

\begin{rem}
For each $t \in \R$, the subset
$M_t\subeq M$ is open, the map 
\[ \Phi_t \: M_t \to M, \quad m \mapsto \Phi(t,m) \] 
is smooth and, for $\Phi_t(m) \in  M_{-t}$, we have 
$\Phi_{-t}(\Phi_t(m)) = m.$ 
\end{rem}

If  $\Phi\colon \cD \to M$ is a local flow, then 
its {\it velocity field} 
$$ X^\Phi(x) := \derat0 \Phi_t(x) = \alpha_x'(0) $$
is a smooth vector field on $M$. 
From \eqref{eq:semprop} 
we immediately obtain for $t \in \R$ 
\begin{equation}
  \label{eq:vecrel}
 T(\Phi_t) \circ X^\Phi\res_{M_t} = X^\Phi \circ \Phi_t.
\end{equation}

\begin{defn}\mlabel{def:1.2} Let $M$ be a smooth manifold.

(a) A smooth vector field $X \in \cV(M)$ is called 
{\it locally integrable} if it is the velocity field 
of some smooth local flow. 

(b) Let $I \subeq \R$ an open interval
containing~$0$. A differentiable map $\gamma\colon I \to M$ is
called an {\it integral
curve} of $X$ if
\begin{equation}
  \label{eq:intcurve}
 \gamma'(t) = X(\gamma(t)) \qquad \hbox{for each} \quad t \in I. 
\end{equation}

(c) 
If $J \supeq I$ is an interval containing $I$, then an integral
curve $\eta \: J \to M$ is called an {\it
extension} of $\gamma$ if
$\eta\res_I = \gamma$. An integral curve $\gamma$ is said to be {\it
maximal} if it has no proper extension.
\end{defn}

Note that \eqref{eq:intcurve} implies that the curve $\gamma'$ in $TM$ is
continuous and further that, if $\gamma$ is $C^k$, then $\gamma'$ is
also $C^k$. Therefore integral curves are automatically smooth.

\begin{ex} \mlabel{ex:countex} 
(cf.\ \cite[Example II.3.11]{Ne06}, [Ham82, 5.6.1]) 
We give examples of linear ODEs on Fr\'echet spaces 
for which multiple solutions to initial value problems exist and for 
which no local solutions exist. 

(a) We consider the Fr\'echet space $E := C^\infty([0,1],\R)$ of smooth 
functions on the closed unit interval, and 
the continuous 
linear operator $Df := f'$ on~$E$. We are asking for solutions of the 
initial value problem 
\begin{equation}
  \label{eq:dotg}
\dot\gamma(t) = D\gamma(t), \quad \gamma(0) = v_0, \quad 
\gamma \: I \subeq \R \to E. 
\end{equation}
It follows from E.~Borel's Theorem \cite{B95} that there exist smooth functions 
$f_+$ on $[1,\infty[$ and $f_-$ on $]-\infty,0]$ such that 
\[ f_+^{(j)}(1) = v_0^{(j)}(1) \quad \mbox{ and } \quad
f_-^{(j)}(0) = v_0^{(j)}(0) \quad \mbox{ for  } \quad j \in \N_0.\] 
Then 
\[ h(x) :=
\begin{cases} 
f_-(x) & \text{ for } x \leq 0 \\
v_0(x) & \text{ for } 0  \leq x \leq 1 \\ 
f_+(x) & \text{ for } x \geq 1   
\end{cases}\]
is a smooth function on $\R$. The curve 
\[  \gamma \: \R \to E, \quad \gamma(t)(x) := h(t + x)  \] 
satisfies $\gamma(0) = h\res_{[0,1]} = v_0$ and 
$\dot\gamma(t)(x) = h'(t + x) = \gamma(t)'(x) = (D \gamma(t))(x)$. It is clear that these 
solutions of \eqref{eq:dotg} 
depend on the choice of the extension $h$ of~$v_0$. 

(b) We now consider the same problem on the space 
$F := C^\infty(]0,1[,\R)$ with $Df = f'$. 
Then the arguments under (a) imply that the 
corresponding initial value problem 
for $v_0(x) := \frac{1}{x-x^2}$ has no solution in any 
open interval containing~$0$.
\end{ex}

As we have just seen, there neither 
is a general result on local existence nor 
uniqueness 
of flows of smooth vector fields on Fr\'echet manifolds. 
Therefore the remarkable 
point of the following lemma is the uniqueness assertion. 

\begin{lem}
  \mlabel{lem:8.6.2.3}
If $X \in \cV(M)$ is a locally integrable vector field, then 
the following assertions hold: 
\begin{itemize}
\item[\rm(i)] If $\Phi \: \cD \to M$ is a local flow on $M$ with 
$X^\Phi = X$, then the 
flow lines are integral curves of $X$. 
\item[\rm(ii)] For each open interval $I \subeq \R$ containing $0$ and 
$x \in M$, there is at most one integral curve $\gamma_x \: I \to M$ 
of $X$. 
\item[\rm(iii)] Two local flows 
$\Phi, \Psi \: \cD \to M$ with $X^\Phi = X^\Psi$ coincide. 
\item[\rm(iv)] Through each point $x \in M$, there exists a unique 
maximal integral curve $\gamma_x \: I_x \to M$ of $X$ with $0 \in I_x$ and 
$\gamma_x(0) = x$. All other integral curves 
$\gamma \: I \to M$ of $X$ with $\gamma(0) = x$ are restrictions of 
$\gamma_x$.
\end{itemize}
\end{lem}

\begin{prf} (i) Let
 $\alpha_x\colon I_x \to M$ be a flow line and $s \in I_x$. For
sufficiently small $t\in \R$, we then have
$$ \alpha_x(s + t) = \Phi_{s+t}(x) = \Phi_t\Phi_s(x) =
\Phi_t(\alpha_x(s)),$$
so that taking derivatives in $t = 0$ leads to
$\alpha_x'(s) = X^\Phi(\alpha_x(s)).$

(ii) Let $\gamma,\eta \: I \to M$ be two integral
curves of $X$ with $\gamma(0) = \eta(0) = p$. The continuity of the
curves implies that
$$0\in J := \{ t \in I\: \gamma(t) = \eta(t)\}$$
is a closed subset of $I$. Suppose that there exists an 
element $0 < t_0 \in I \setminus J$. 
Then $s := \sup J < t_0$ and $x := \gamma(s) = \eta(s)$. 
We now consider the curve 
\[ \tilde\gamma(t) := \Phi_{-t}(\gamma(t+s)),\] 
defined on a small interval $I'$ containing $0$. 
Clearly, $\tilde\gamma(0) = \gamma(s) = \alpha_s(x)$. Moreover, we obtain 
with \eqref{eq:vecrel} 
\begin{align*}
\tilde\gamma'(t) 
&= -X(\tilde\gamma(t)) + T(\Phi_{-t})\gamma'(t+s)
= -X(\tilde\gamma(t)) + X(\Phi_{-t}(\gamma(t+s)))\\
&= -X(\tilde\gamma(t)) + X(\tilde\gamma(t)) = 0.
\end{align*}
This implies that $\tilde\gamma$ is constant (here we use the local 
convexity of $M$), so that 
$\Phi_{-t}(\gamma(t+s)) = x$ for 
$t \in [0,\eps]$ and some $\eps > 0$. 
If $\eps$ is sufficiently small, then 
$\Phi_{-t}$ is defined on $\gamma(t+s)$ for $0 \leq t \leq \eps$, 
and applying $\Phi_t$ on both sides leads to 
$\gamma(t+s) = \alpha_t(x)$. The same argument shows that 
$\eta(t+s) = \alpha_t(x)$ for $0 \leq t$ sufficiently small. 
This contradicts the definition 
of $s$ and thus shows that $\sup J = \sup I$. 
We likewise obtain $\inf J = \inf I$, and therefore $J = I$. 

(iii) follows immediately from (i) and (ii). 

(iv) If $\gamma \: I \to M$ and $\eta \: J \to M$ are 
integral curves of $X$ with $\gamma(0) = \eta(0) = x$, 
then (ii) implies that $\gamma\res_{I \cap J} = \eta\res_{I \cap J}$, 
so that both curves combine to an integral curve on the 
open interval $I \cup J$. 

Let $I_x \subeq \R$ be the union of all open intervals $I_j$
containing $0$ on
which there exists an integral curve $\gamma_j \: I_j \to M$ of $X$
with $\gamma_j(0) = x$. Then the preceding argument shows that
$$ \gamma(t) := \gamma_j(t) \quad \hbox{ for } \quad t \in I_j $$
defines an integral curve of $X$ on $I_x$, which is maximal by
definition. The uniqueness of the maximal integral curve 
follows from its definition.
\end{prf}

\begin{thm}
  \mlabel{thm:8.6.2.4}
Each locally integrable smooth vector field 
$X$ is the velocity field of a unique 
local flow defined by
$$ {\cal D}_X  := \bigcup_{x \in M} I_x \times \{x\} \quad\hbox{ and }\quad \Phi_t(x) :=
\gamma_x(t) \quad \hbox{ for } \quad (t,x) \in {\cal D}_X, $$
where $\gamma_x \: I_x \to M$ is the unique maximal integral curve through $x \in M$.
\end{thm}

\begin{prf} If $(s,x), \big(t,\Phi_s(x)\big)$ and $(s+t,x) \in {\cal D}_X$, the relation
$$ \Phi_{s+t}(x) = \Phi_t \Phi_s(x) \quad \hbox{ and } \quad
I_{\Phi_s(x)} = I_{\gamma_x(s)} = I_x - s $$
follow from the fact that both curves
$$ t \mapsto \Phi_{t+s}(x)= \gamma_x(t+s) \quad\hbox{ and }\quad t \mapsto \Phi_t\Phi_s(x)
 = \gamma_{\Phi_s(x)}(t) $$
are integral curves of $X$ with the initial value $\Phi_s(x)$, hence coincide.

We claim that all maps
\[ \Phi_t \: M_t = \{ m\in M \: (t,m) \in \cD\} \to M,
\quad x \mapsto \Phi_t(x)\] 
are injective. In fact, if $p :=
\Phi_t(x) = \Phi_t(y)$, then $\gamma_x(t) = \gamma_y(t)$, and on
$[0,t]$ the curves $s \mapsto \gamma_x(t-s), \gamma_y(t-s)$ are
integral curves of $-X$, starting in $p$. Hence 
Lemma~\ref{lem:8.6.2.3}(ii) implies that they coincide in $s = t$,
which means that $x = \gamma_x(0)  = \gamma_y(0) = y$. From this
argument it further follows that $\Phi_t(M_t) = M_{-t}$ and
$\Phi_t^{-1} = \Phi_{-t}$.

It remains to show that ${\cal D}_X$ is open and $\Phi$ smooth.
The local integrability of $X$ 
provides for each $x \in M$ an open neighborhood
$\cD_x$ and some $\eps_x > 0$, as well as a smooth map
$$ \phi_x\colon {}]-\eps_x, \eps_x[\, \times \cD_x \to M, \quad
\phi_x(t,y) = \gamma_y(t) = \Phi_t(y). $$
Hence
$]-\eps_x, \eps_x[\, \times \cD_x  \subeq {\cal D}_X$, and the restriction of $\Phi$
to this set is smooth.
Therefore $\Phi$ is smooth on a neighborhood of $\{0\} \times M$ in ${\cal D}_X$.

Now let $J_x$ be the set of all $t \in [0,\infty[$, for which
${\cal D}_X$ contains a neighborhood of $[0,t] \times \{x\}$ on which $\Phi$ is smooth.
The interval $J_x$ is open in $\R_+ := [0,\infty[$ by definition. We claim that
$J_x = I_x\cap \R_+$. This entails that ${\cal D}_X$ is open because the same argument applies to
$I_x \cap \,]-\infty,0]$.

We assume the contrary and find a minimal $\tau \in I_x \cap \R_+
\setminus J_x$, because this interval is closed. Put $p :=
\Phi_\tau(x)$ and pick a product set $I \times W \subeq {\cal D}_X$,
where $W$ is an open neighborhood of $p$ and $I = \,]-2\eps,2\eps[$ a
$0$-neighborhood, such that $2\eps < \tau$ and $\Phi : I \times W
\to M$ is smooth. By assumption, there exists an open neighborhood $V$ of
$x$ such that $\Phi$ is smooth on $[0,\tau-\eps{}] \times V \subeq
{\cal D}_X$. Then $\Phi_{\tau -\eps}$ is smooth on $V$ and
$$V' := \Phi_{\tau-\eps}^{-1}\big(\Phi_\eps^{-1}(W)\big) \cap V $$
is a neighborhood of $[0,\tau+\eps] \times \{x\}$ in ${\cal D}_X$.
Further,
$$V' = \Phi_{\tau-\eps}^{-1}\big(\Phi_\eps^{-1}(W)\big) \cap V
= \Phi_\tau^{-1}(W) \cap V, $$
and $\Phi$ is smooth on $V'$, because it is a composition of smooth maps:
\[  ]\tau - 2\eps, \tau + 2\eps[ \times V' \to M, \quad (t,y) \mapsto
\Phi_{t-\tau}\Phi_\eps\Phi_{\tau-\eps}(y). \]
We thus arrive at the contradiction $\tau \in J_x$.

This completes the proof of the openness of ${\cal D}_X$ and the smoothness
of~$\Phi$.
The uniqueness of the flow follows from the uniqueness of the integral 
curves (Lemma~\ref{lem:8.6.2.3}(ii)). 
\end{prf}

\begin{rem}
  \mlabel{rem:8.6.2.6}
Let $\Phi^X \: {\cal D}_X \to M$ be the maximal local flow of a
locally integrable 
vector field $X$ on $M$. Let $M_t = \{x \in M \: (t,x) \in {\cal
D}_X\}$, and observe that this is an open subset of $M$. We have
already seen in the proof of Theorem~\ref{thm:8.6.2.4} above that
the smooth maps $\Phi^X_t \: M_t \to M$ are injective with
$\Phi^X_{t}(M_{t}) = M_{-t}$ and $(\Phi^X_{t})^{-1} = \Phi^X_{-t}$
on the image. It follows in particular that $\Phi^X_{t}(M_t) =
M_{-t}$ is open and that
$\Phi^X_t \: M_t \to M_{-t}$
is a diffeomorphism whose inverse is $\Phi^X_{-t}$.
\end{rem}

\begin{rem} Suppose that 
$X \in \cV(M)$ is locally integrable and that 
$\phi \: M \to N$ is a diffeomorphism. 
Then the vector field $Y := \phi_*X$ on $N$ is also locally integrable 
and its local flow satisfies 
\[ \Phi^Y_t \circ \phi\res_{M_t} = \phi \circ \Phi^X_t\res_{M_t}.\] 
\end{rem}

\begin{prop} \mlabel{prop:lin} 
Let $V$ be a locally convex vector space and 
$D \: V \to V$ be a continuous linear map. 
If the corresponding smooth vector field 
$X(v) = Dv$ is locally integrable, then the corresponding local flow is 
global. 
\end{prop}

\begin{prf} Suppose that $X$ is locally integrable. Then the domain 
$\cD \subeq \R \times V$ of its local flow contains a neighborhood 
$[-\eps,\eps] \times U_V$ of $(0,0)$. The linearity of the vector field 
further implies that each subset $V_t$, $t\in \R$, is a linear subspace of $V$, 
hence $V_t = V$ whenever $V_t \not=\eset$. We conclude that 
$[-\eps,\eps] \times V \subeq \cD$, and this implies that the flow it global. 
\end{prf}

\section{Lie Derivatives}
\mlabel{sec:2}

In this section we take a closer look at the interaction of local flows and vector
fields. Let $X \in {\cal V}(M)$ be locally integrable 
and let $\Phi^X \: {\cal D}_X \to M$ be its maximal
local flow. For  $f\in C^\infty(M)$ and $t\in \R$, we set
$$(\Phi^X_t)^*f:=f\circ \Phi^X_t\in C^\infty(M_t).$$
Then 
$$\lim_{t \to 0} \frac{1}{t}( (\Phi^X_{t})^*f-f)=\dd f(X)=\cL_Xf\in
C^\infty(M).$$ For a second vector field $Y \in {\cal V}(M)$, we
define a smooth vector field on the open subset $M_{-t} \subeq M$ by
$$ (\Phi^X_t)_* Y
:= T(\Phi^X_t) \circ Y \circ \Phi^X_{-t} = T(\Phi^X_t) \circ Y \circ
(\Phi^X_t)^{-1} $$ (cf.\ Remark~\ref{rem:8.6.2.6}) and define the
{\it Lie derivative} by 
\[  {\cal L}_X Y
:= \lim_{t \to 0} \frac{1}{t}( (\Phi^X_{-t})_* Y-Y)
= \derat0 (\Phi^X_{-t})_* Y, \]
which is defined on all of $M$ since, for each $p \in M$, the vector
$((\Phi^X_t)_*Y)(p)$
is defined for sufficiently small $t$ and depends smoothly
on $t$.

\begin{thm}
  \mlabel{thm:8.6.2.7}
${\cal L}_X Y = [X,Y]$ for $X,Y \in {\cal V}(M)$.
\end{thm}

\begin{prf} Fix $p \in M$. It suffices to show that ${\cal L}_X Y$ and $[X,Y]$ coincide in $p$.
We may therefore work in a local chart, hence assume that
$M = U$ is an open subset of some locally convex space~$E$. 

Identifying vector fields with smooth $E$-valued functions, we have
$$ [X,Y](x) = \dd Y(x) X(x) - \dd X(x) Y(x)\quad \mbox{ for } 
\quad x \in U. $$
On the other hand,
\begin{align*}
&((\Phi^X_{-t})_* Y)(x)
= T(\Phi^X_{-t}) \circ Y \circ \Phi^X_{t}(x) \cr
&= \dd (\Phi^X_{-t})(\Phi^X_{t}(x)) Y(\Phi^X_{t}(x))
= \big(\dd (\Phi^X_{t})(x)\big)^{-1} Y(\Phi^X_{t}(x)).
\end{align*}
To calculate the derivative of this expression with respect
to $t$, we first observe that it does not matter if we first
take derivatives with respect to $t$ and then with respect to $x$ or
vice versa. This leads to
$$ \derat0 \dd (\Phi^X_t)(x) =
\dd\Big(\derat0 \Phi^X_t\Big)(x) = \dd X(x). $$
Next we note that
for any curve
$\alpha \: [-\eps,\eps] \to \GL(V)$ with $\alpha(0) = \1$ 
for which the associated map 
\[ \hat\alpha \: [-\eps,\eps] \times V \to V \times V, \quad 
\hat\alpha(t,v) := (\alpha(t)v, \alpha(t)^{-1}v) \] 
is smooth, we  have
$$ (\alpha^{-1})'(t) = - \alpha(t)^{-1} \alpha'(t) \alpha(t)^{-1}, $$
and in particular $(\alpha^{-1})'(0) = - \alpha'(0)$.
Combining all this, we obtain with the Product Rule
\[{\cal L}_X(Y)(x)
=  -\dd X(x)Y(x) + \dd Y(x)X(x)
= [X,Y](x). \qedhere \]
\end{prf}

\begin{ex} If $V = C^\infty_c(]0,1[)$ and $Df = f'$, then for every 
$f \in V$, there exists a smooth integral curve 
$\gamma \: [-\eps,\eps] \to V$ with 
$\dot\gamma(t) = D \gamma(t)$, given by 
$\gamma(t)(s) = f(t+s)$ for $\supp(f) - t \subeq ]0,1[$, 
but the corresponding linear vector field $\bx(v) = Dv$ is not integrable 
since the corresponding flow is not defined on a neighborhood of $(0,0)$ 
in $\R \times V$ (cf.\ Proposition~\ref{prop:lin}). 
\end{ex}

\begin{defn}\mlabel{def:2.4}
A locally convex Lie algebra $\g$ is call {\it $\ad$-integrable} if 
all linear vector fields $X_x(y) :=[x,y]$  generate 
global flows. If $\g$ is $\ad$-integrable, then we set
\[e^{\ad x}:= \Phi_1^{X_x} \quad \mbox{ for } \quad x\in \g\, .\]
\end{defn}

\begin{lem}
  \mlabel{lem:vectrafo}
Let $\g$ be a locally convex
$\ad$-integrable Lie algebra. 
Let $\beta \: \g \to \cV(M)$ be a homomorphism of Lie algebras 
for which the corresponding map 
\[ \hat\beta \: \g \times M \to TM, \quad (x,m) \mapsto \beta(x)(m)\] 
is smooth. 
Suppose that $\beta(x)$ is locally integrable and write 
$\Phi^x$ for the corresponding local flow on $M$. 
Then 
\begin{equation}
  \label{eq:trafo}
(\Phi^x_{-t})_* \beta(y) = \beta(e^{t\ad x}y)\res_{M_{t}} 
\quad \mbox{ for } \quad y \in \g, t \in \R.
\end{equation}
\end{lem}

\begin{prf} Pick $m \in M$ and assume that $\Phi^x_{-t}(m)$ is defined 
in $t_0$. For the curve 
\[ \gamma(t) := \big((\Phi_t^x)_*\beta(e^{t\ad x}y)\big)(m) 
= T(\Phi^x_t) \beta(e^{t \ad x}y)(\Phi^x_{-t}(m)) \] 
our smoothness assumption on $\beta$ implies with 
Theorem~\ref{thm:8.6.2.7} that 
\begin{align*}
\gamma'(t_0) 
&= \derat{t_0} \big((\Phi_t^x)_*\beta(e^{t_0\ad x}y)\big)(m) 
+ \derat{t_0} \big((\Phi_{t_0}^x)_*\beta(e^{t\ad x}y)\big)(m) \\
&= \Big((\Phi^x_{t_0})_* \cL_{-\beta(x)} \beta(e^{t_0\ad x}y)\Big)(m) 
+ \big((\Phi_{t_0}^x)_*\beta([x, e^{t_0\ad x}y])\big)(m) \\
&= \Big( (\Phi^x_{t_0})_*\big(
[-\beta(x),\beta(e^{t_0\ad x}y)]+ \beta([x, e^{t_0\ad x}y])\big)\Big)(m) 
=0.
\end{align*}
We conclude that $\gamma$ is constant, and hence that, 
whenever $\Phi^x_{-t}(m)$ is defined, we have 
\[ T(\Phi^x_t)\beta(e^{t\ad x}y)(\Phi^x_{-t}(m)) = \beta(y)(m).\] 
This leads to 
\[ \beta(e^{t\ad x}y)(m) = T(\Phi^x_{-t})\beta(y)(\Phi^x_t(m)) 
= \big((\Phi^x_{-t})_*\beta(y)\big)(m) \quad \mbox{ for } \quad m \in M_t,\] 
so that \eqref{eq:trafo} is proved. 
\end{prf}

\section{Vector fields and positive definite kernels} 
\mlabel{sec:2b}

In this section $M$ is a smooth locally convex manifold 
and we study the interaction between a
 smooth positive definite kernel $K \: M \times M \to \C$ and vector fields. 
Our main result is the Geometric Fr\"ohlich Theorem (Theorem~\ref{T:Froelichgeo})  
which implies in particular that for a locally integrable smooth vector field 
$X \in \cV(M)$ the Lie derivative defines a selfadjoint operator on 
the corresponding reproducing kernel space~$\cH_K$.

\begin{defn} Let $K \in C^\infty(M \times M,\C)$ be a smooth positive definite 
kernel. A vector field $X \in \cV(M)$ is said to be 
{\it $K$-symmetric} if 
\[ \cL_X^1 K = \cL_X^2 K \] 
and 
{\it $K$-skew-symmetric} if 
\[ \cL_X^1 K = -\cL_X^2 K. \] 
Here the superscripts indicate whether the Lie derivative acts on the 
first or the second argument. 
\end{defn}

In the following we assume that $K$ is a smooth positive definite 
kernel on $M$ and write $\cH_K \subeq \C^M$ for the corresponding reproducing 
kernel Hilbert space (see Appendix~\ref{app:d} for details). 
Then the functions $K_m := K(\cdot,m)$, $m \in M$, span a dense subspace 
$\cH_K^0$ and we have 
\[ f(m) = \la f, K_m \ra \quad \mbox{ for } \quad f\in \cH_K, m \in M\] 
(cf.\ Definition~\ref{def:1.5}). 
According to \cite[Thm.~7.1]{Ne10}, the smoothness of $K$ implies that the map 
\[ M\to \cH_K,\quad m\to  K_m\, ,\]
is smooth, so that $\phi(m) = \la \phi, K_m\ra$ 
for $\phi \in \cH_K$ implies that $\cH_K \subeq C^\infty(M,\C)$. 
The Lie derivative defines a representation of $\cV(M)$ on $C^\infty(M,\C) \supeq \cH_K$. 
For every smooth vector field $X \in \cV(M)$, we therefore define  
\[ \cD_X := \{ \phi \in \cH_K \: \cL_X \phi \in \cH_K\}, \quad 
\cL_X^K := \cL_X\res_{\cD_X} \: \cD_X \to \cH_K.\] 

\begin{lem} \mlabel{lem:ode} 
If $\gamma \: [a,b] \to M$ is an integral curve of $X$ and 
$\cL_X^1 K = \eps \cL_X^2 K$, $\eps \in \{\pm 1\}$, then the curve
 $\eta(t) := K_{\gamma(t)}$ in $\cH_K$ is smooth, contained in $\cD_X$, 
and satisfies the differential equation 
\begin{equation}
  \label{eq:diffeq}
\eta'(t) = \eps \cL_X^K \eta(t).
\end{equation}
\end{lem}

\begin{prf} The smoothness of the map $M \to \cH_K, m \mapsto K_m$, and the smoothness 
of $\gamma$ imply that $\eta$ is smooth. We further have, for $m \in M$, 
\begin{align*}
\eta'(t)(m) 
&= \la  \eta'(t), K_m \ra 
= \frac{d}{dt} K(m, \gamma(t)) 
= \cL_X^2 K(m, \gamma(t)) \\
&=  \eps \cL_X^1 K(m, \gamma(t)) 
= \eps \cL_X K_{\gamma(t)}(m).
\end{align*}
This implies \eqref{eq:diffeq}. 
We conclude in particular that $\eta(t) \in \cD_X$, and the assertion follows. 
\end{prf}

\begin{prop} \mlabel{prop:graphclosed} Suppose that $X \in \cV(M)$ satisfies 
$\cL_X^1 K = \eps \cL_X^2 K$ for $\eps \in \{\pm 1\}$ and that 
integral curves through every $m \in M$ exist. 
Then the following assertions hold:  
\begin{enumerate}
 \item[\rm(i)] $\cH_K^0 \subeq \cD_X$.
 \item[\rm(ii)] If $X$ is symmetric (resp. skew-symmetric), then $\cL_X^K|_{\cH_k^0}$ is
 a symmetric (resp. skew-symmetric) operator.
\item[\rm(iii)] $\cL_X^K$ is a closed operator.
\item[\rm(iv)] If $\cL_X^K|_{\cH_k^0}$ is
 essentially selfadjoint (resp.~skew-adjoint), then $\cL_X^K=\oline{\cL_X^K|_{\cH_k^0}}$.
 \end{enumerate}
\end{prop}
 
\begin{prf}(i) The relation $\cH_K^0 \subeq \cD_X$ follows directly 
from Lemma~\ref{lem:ode}.\\ 
(ii) Since $K(m,n)=\oline{K(n,m)}$, we have
\begin{align*}
\la\cL_X K_m,K_n\ra & =\cL^1_X K(n,m)=\eps \cL_X^2 K(n,m)
=\eps\oline{\cL_X^1 K(m,n)}\\
&=\eps\oline{\la\cL_X K_n,K_m\ra}=
\eps\la K_m,\cL_X K_n\ra. 
\end{align*}
(iii) To see that $\cL_X^K$ is closed, assume that 
$(f_n, \cL_X^K f_n) \to (f,g)$ in $\cH_K\times \cH_K$. Then we have, for every $m \in M$, any 
integral curve $\gamma_m$ through $m$, and $\phi \in \cH_K$, the relation 
\begin{equation}\label{eq:opclosed}
(\cL_X \phi)(m) 
= \derat0 \phi\circ \gamma_m 
= \derat0 \la \phi, K_{\gamma_m} \ra 
= \eps \la \phi, \cL_X^K K_m \ra.
\end{equation}

In particular, $f_n \to f$ implies that 
$\cL_X f_n \to \cL_X f$ pointwise on $M$, and since 
we also have 
$(\cL_X^K f_n)(m) = \la \cL_X^K f_n, K_m \ra \to \la g, K_m \ra = g(m)$, 
it follows that $g = \cL_X f$. This implies that 
$f \in \cD_X$ and $g = \cL_X^K f$.\\ 
(iv) Let $T=\oline{\cL_X^K|_{\cH_k^0}}$. Then our assumption implies 
$T=\eps (\cL_X^K|_{\cH_k^0})^*$. 
Since $\cL_X^K$ is closed, we have $T\subeq  \cL_X^K$.
Conversely,  \eqref{eq:opclosed} shows that
$\cL_X^K \subeq \eps (\cL_X^K\res_{\cH_K^0})^*$. Hence $\cL_X^K=T$. 
\end{prf}

\begin{lem} \mlabel{lem:3.2} Suppose that $X \in \cV(M)$ has the 
property there exists a smooth $\R$-action 
$\sigma \: \R\times M \to M$ for which the orbit curves 
$t\mapsto \sigma_t(m):=\sigma (t,m)$ are integral curves of $X$. 
If $X$ is $K$-skew-symmetric, then 
\[ U_t \phi := \phi \circ \sigma_t \] 
defines a continuous unitary representation on $\cH_K$ with 
\begin{equation}
  \label{eq:kerinv}
  U_t K_m = K_{\sigma_{-t}(m)}. 
\end{equation}
In particular, $\cH_K^0$ consists of smooth vectors for $U$. 
Moreover, the infinitesimal generator of $U$ coincides with 
$\cL_X^K$ and $\cH_K^0$ is a core for $\cL_X^K$.
\end{lem}

\begin{prf} The invariance of $K$ under the flow $\sigma$ follows immediately 
from the $K$-skew-symmetry of $K$, which shows that, for $m,n \in M$, the functions 
$t \to K(\sigma_t(m), \sigma_t(n))$ are constant. This implies that 
$K_m \circ \sigma_t = K_{\sigma_{-t}(m)}$, which is \eqref{eq:kerinv}. 
The invariance of $\cH_K$ under $\sigma$ and that $U$ is a 
unitary representation follows from \cite[Prop.~II.4.9]{Ne00}. 
As $\cH_K^0$ is $U$-invariant, Stone's Theorem 
implies that the infinitesimal generator $T$ of $U$ coincides with the closure 
of $T\res_{\cH_K^0}$. Since on $\cD(T)$, and in particular
on $\cH_K^0$, $T$ and $\cL_X^K$ coincide, the last assertion 
follows from Proposition~\ref{prop:graphclosed}(iv).
\end{prf}

Below we recall Fr\"ohlich's Theorem on unbounded symmetric semigroups
as it is stated in \cite[Cor.~1.2]{Fro80} (see also \cite{MN12}).
Actually Fr\"ohlich assumes that the Hilbert space $\cH$ is separable, 
but this is not necessary for the conclusion. Replacing the assumption 
of weak measurability by weak continuity,  
all arguments in \cite{Fro80} work for non-separable spaces as well. 

\begin{thm}{\rm(Fr\"ohlich)}  \mlabel{thm:2.4} 
Let $H$ be a symmetric operator defined on the dense subspace $\cD$ of 
the Hilbert space $\cH$. Suppose that, for every $v\in\cD$, there
exists an $\eps_v>0$ and a curve 
$\phi \: ]0,\eps_v[ \to \cD$ satisfying 
\[ \dot\phi(t)=H\phi(t) \quad \mbox{ and }\quad 
\lim_{t\to 0}\phi(t)=v.\] 
Then the operator $H$ is essentially 
self-adjoint and $\phi(t) = e^{t\oline H}v$ in the sense of 
spectral calculus.
\end{thm}

In \cite[Rem.~2.6]{MN12} we applied Fr\"ohlich's Theorem 
to the action of the Lie derivative on a Hilbert space with a smooth 
reproducing kernel. This leads to the following theorem. 

\begin{thm} {\rm(Geometric Fr\"ohlich Theorem)} 
\mlabel{T:Froelichgeo}
Let $M$ be a locally convex manifold and $K$ be a smooth positive definite 
kernel. If $X$ is a $K$-symmetric vector 
field on $M$ with the property that, for every $m\in M$, there exists an 
integral curve $\gamma_m \: [0,\eps_m] \to M$ starting in $m$, 
then $\cL_X\res_{\cH_K^0}$ is an essentially selfadjoint 
operator $\cH_K^0 \to \cH_K$ whose closure coincides with 
$\cL_X^K$. For $0 \leq t \leq \eps_m$, we have 
\[ e^{t\cL_X^K} K_m = K_{\gamma_m(t)}.\] 
\end{thm}

\begin{prf} First we recall that the $K$-symmetry of $X$ implies that 
$H := \cL_X^K\res_{\cH_K^0}$ is a symmetric operator (Proposition~\ref{prop:graphclosed}(ii)). 
Lemma~\ref{lem:ode} further implies that, for every integral curve $\gamma$ of $X$, the 
curve $\eta(t) :=K_{\gamma(t)}$ in $\cH_K$ satisfies the ODE 
$\eta'(t) = H \eta(t)$. Therefore Theorem~\ref{thm:2.4} applies and $\cL_X^K=\oline{H}$ follows
from Proposition~\ref{prop:graphclosed}(iv).
\end{prf}

\begin{ex} The simplest but typical class 
of examples where Fr\"ohlich's Theorem applies 
arise from smooth functions $\phi \: ]a,b[ \to \C$ on an open interval in 
$\R$ for which the kernel $K(x,y) := \phi\big(\frac{x+y}{2}\big)$ 
is positive definite. 
Then the vector field $X = \frac{\partial}{\partial t}$ 
satisfies $\cL_X f = f'$ on $\cH_K^0$, so that the operator 
$Df := f'$ on its maximal domain in $\cH_K$ is selfadjoint 
by the preceding theorem. 

To see the corresponding unitary one-parameter group 
$U_t := e^{itD}$, one can use \cite{Sh84} to see that 
the function $\phi$ extends analytically to the strip 
\[ S := \{ z \in \C \: a < \Re z < b\},\] so that 
$K^\C(z,w) := \phi\big(\frac{z+\oline w}{2}\big)$ provides 
an extension of the kernel $K$ to $S \times S$. On the corresponding 
Hilbert space $\cH_{K^\C} \subeq \cO(S)$, we then have 
$(U_t f)(z) = f(z+ it)$. 
\end{ex}

\begin{ex} For later use we record here how Fr\"ohlich's Theorem applies in the linear case. 
Let $V$ be a locally convex space
and $\la\cdot,\cdot\ra \: V \times V \to \C$ be 
a continuous positive semidefinite hermitian 
form. Then $K(v,w)=\la w,v\ra$ is a positive definite kernel and $\cH_K$ 
is identified with a subspace of the $V^\sharp$ of 
antilinear continuous functionals. 
The continuity of the kernel $K$ implies that the antilinear map 
$V\rightarrow \cH_K,\ v\mapsto K_v$ is continuous. 
For any continuous operator $L \: V \to V$, the formula
$$L\lambda:=-\lambda\circ L$$
defines by restriction to $\cD_L:=\{\lambda\in V^\sharp\mid L\lambda\in\cH_K\}$
an unbounded closed operator $L^K:\cD_L\rightarrow\cH_K$. 
If there exists $L^*:V\rightarrow V$ with
$$\la v,Lw\ra=\la L^*v,w\ra\quad\text{for}\quad v,w\in V$$
then
\begin{equation}\label{E:8}
LK_v=K_{-L^*v}\quad\text{for}\quad v\in V.
\end{equation}
\end{ex}

We can now formulate the following corollary to Theorem~\ref{thm:2.4}: 

\begin{cor} \label{T:3.7}
Let $L \: V \to V$ be a symmetric linear operator 
with the property that, for every $v\in V$, there exists a  
curve $\gamma_v \: [0,\eps_v] \to V$ starting in $v$ and satisfying the differential 
equation 
\[ \gamma_v'(t) = L \gamma_v(t).\] 
Then $L^K|_{\cH_K^0}$ 
is an essentially selfadjoint operator whose closure 
coincides with~$L^K$. For $0 \leq t \leq \eps_v$, we have 
\[ e^{-t L^K} K_v = K_{\gamma_v(t)}.\] 
\end{cor}


\section{Integrability for reproducing kernel spaces} 
\mlabel{sec:3}

In this section we study Lie algebra actions on the locally convex manifold $M$
that are compatible with the kernel $K$. Our first main result is Theorem~\ref{thm:4.8} 
which provides a sufficient condition for the Lie algebra 
representation of $\g^c$ coming from an action of $\g$ on $\cH_K$ by Lie derivatives 
to integrate to a unitary representation of the corresponding simply connected Lie group~$G^c$.
Applying this result to open subsemigroups of Lie groups, we further obtain 
an interesting generalization of the L\"uscher--Mack Theorem for semigroups that need not 
possess a polar decomposition.

\subsection{Smooth right actions and compatible kernels}

\begin{defn} \mlabel{def:3.9}
Let $\fg$ be a Lie algebra and $\tau$ be an involutive automorphism of $\fg$. The
pair $(\fg,\tau)$ is called a \emph{symmetric Lie algebra}. If $\fh:=\ker(\tau-\1)$ and  
$\fq=\ker(\tau + \1)$, then $[\h,\h]\subseteq \h$, $[\fq,\fq]\subseteq \h$
and $[\fh,\fq]\subseteq \fq$. 
It follows that $\g^c:=\h+i\fq$ is a Lie algebra.
It is called the \textit{Cartan dual} of $\g=\h+\fq$. Extending $\tau$ to
a complex linear automorphism of $\g_\C$ and then restricting to
$\g^c$ shows that $(\g^c,\tau)$ is also a symmetric Lie algebra.
\end{defn}

\begin{ex} Let $d\in \N$ and $p,q\in \N_0$ such that $p+q=d$. 
Let 
\[I_{p,q}=\begin{pmatrix} -I_p & 0 \\ 0 & I_q\end{pmatrix}\, .\]
On $G=\R^d\rtimes \OO_d(\R)$ we consider the involution
\[\tau (x,a)=(I_{p,q}x,I_{p,q}aI_{p,q})\, .\]
Then 
\[\fq=(\R^p\oplus 0_q)\oplus \left\{\left. \begin{pmatrix} 0 & X\\ -X^\top & 0\end{pmatrix}
\, \right|\, X\in M_{p,q}\right\}\, .\]
It is then easy to see that
\[\fg^c\simeq \R^{p,q}\rtimes \fo_{p,q}(\R)\, .\]
Hence $G^c$ is locally isomorphic to $\R^{p,q}\rtimes \OO_{p,q}(\R)$. 
Here $\R^{p,q} \cong  i \R^p \oplus \R^q$ stands for
the vector space $\R^d$ with the bilinear form 
\[ \beta_{p,q}(x,y)=x_1y_1+\ldots + x_py_p - x_{p+1}y_{p+1}-\ldots - x_ny_n.\] 
\end{ex}

\begin{defn} Let $(\g,\tau)$ be a symmetric Lie algebra, and let
$\beta \: \g \to \cV(M)$ be a homomorphism. 
A smooth positive definite kernel  
$K \in C^\infty(M \times M,\C)$ is said to be 
{\it $\beta$-compatible} if the vector fields 
in $\beta(\fh)$ are $K$-skew-symmetric and the vector fields 
in $\beta(\fq)$ are $K$-symmetric. 
\end{defn}

Let $H$ be a connected Lie group with Lie algebra $\fh$.

\begin{defn}
  \mlabel{def:3.1} 
A {\it smooth right action} of the pair $(\fg,H)$ on a 
(locally convex) manifold $M$ is a pair $(\beta,\sigma)$, 
where 
\begin{itemize}
\item[(a)] $\sigma \: M \times H \to M$ is a 
smooth right action, 
\item[(b)] and $\beta \: \g \to \cV(M)$ is a homomorphism 
of Lie algebras for which
\[ \hat\beta \: \g \times M \to TM,\quad 
(x,m) \mapsto \beta(x)(m) \] 
is smooth,  
\item[(c)] $\dot\sigma(x) = \beta(x)$ for $x \in \fh$, 
\item[(d)] each vector field $\beta(x)$, $x \in \fq$, is locally integrable. 
\end{itemize}
\end{defn}

\begin{rem} \mlabel{rem:3.2} 
In view of Lemma~\ref{lem:vectrafo}, conditions 
(a)-(d) above imply that 
\[ (\sigma_{\exp x})_* \beta(y) = \beta(e^{-\ad x}y) 
\quad \mbox{ for } \quad x \in \fh, y \in \g.\] 
\end{rem}

In the following $K$ is a smooth $\beta(\g)$-compatible positive definite 
kernel on $M \times M$. 
For $x \in \g$, we abbreviate $\cL_x := \cL_{\beta(x)}^K$ for the maximal restriction 
of the Lie derivatives to $\cD_x:=\cD_{\beta(x)}$ and we extend this definition to 
$\fg_\C$ in an obvious fashion. We moreover consider the subspaces 
\[ \cD^1 := \bigcap_{x \in \g} \cD_{x} 
= \{ \phi \in \cH_K \: (\forall x \in \fg)\, \cL_{\beta(x)}\phi \in \cH_K\}\] 
and 
\[ \cD := \{ \phi \in \cH_K \: (\forall n \in \N)
(\forall x_1,\ldots, x_n \in \fg)\, \cL_{\beta(x_1)}\cdots 
\cL_{\beta(x_n)}\phi \in \cH_K\}\, . \]

Then
$$\alpha:\fg_\C \rightarrow \End(\cD),\ x\mapsto \cL_x|_\cD$$ 
defines a Lie algebra representation such that  $\fg^c$ acts by skew-symmetric operators.
Our main goal in this section is to prove integrability of $\alpha|_{\g^c}$.


First we show that $\cD$ is a dense subspace of $\cH_K$.

\begin{prop}
  \mlabel{lem:3.5} 
We have $\cH_K^0\subeq\cD$. In particular $\cD$ is dense in $\cH$.
\end{prop}

\begin{prf} Let $x_1, \ldots, x_n \in \fh \cup \fq$, $m \in M$ and choose $\eps > 0$ 
so small that we have a smooth map 
\[ \gamma \: [0,\eps]^n \to M,\ \gamma(t_1,\dots,t_n)
=\Phi_{t_1}^{\beta(x_1)}\circ\dots\circ \Phi_{t_n}^{\beta(x_n)}(m)\] 
Then the curves 
\[ t \mapsto \gamma(0,\ldots, 0, t, t_{k+1}, \ldots, t_n) \] 
are integral curves of $\beta(x_k)$. 
Let $\eta(t_1,\dots,t_n) := K_{\gamma(t_1,\dots,t_n)}$. Then $\eta \: [0,\eps]^n\to \cH_K$ is smooth and 
Lemma~\ref{lem:ode} implies that 
\[ \frac{\partial}{\partial t_k} \eta(0,\ldots, 0, t_k, \ldots, t_n) 
= \eps_k \cL_{x_k} \eta(0,\ldots, 0, t_k, \ldots, t_n),\] 
where $\cL_{\beta(x_k)}^1 K = \eps_k \cL_{\beta(x_k)}^2 K$. We conclude with the closedness of the 
operators $\cL^K_{\beta(x_j)}$ (Proposition~\ref{prop:graphclosed}) 
that $K_{\gamma(0)} \in \cD$ with 
\[ \frac{\partial^n \eta}{\partial t_1\cdots \partial t_n}(0)
=  (\eps_1 \cdots \eps_n) \cL_{x_1} \cdots \cL_{x_n} K_{\gamma(0)}.\] 
This shows that $\cH_K^0 \subeq \cD$. 
\end{prf}

From Theorem~\ref{T:Froelichgeo} and Proposition~\ref{lem:3.5}, we obtain:

\begin{cor}\label{C:esssa}
For $x\in\fq$, the operator
$\cL_x$ is selfadjoint and $\cH_K^0$, and hence also $\cD$, is a core for $\cL_x$.
\end{cor}

\noindent The following lemma follows from Lemma~\ref{lem:3.2}. 

\begin{lem} \mlabel{lem:3.1} If $K$ is a smooth $\beta$-compatible positive definite kernel, 
then $K$ is $H$-invariant and 
\[ \pi^H(h)\phi := \phi \circ \sigma_h \] 
defines a continuous unitary representation $(\pi^H, \cH_K)$ of $H$, for which 
\begin{equation}
  \label{eq:kerinv2}
  \pi^H(h)K_m = K_{m.h^{-1}} \quad \mbox{ for } \quad m \in M, h \in H.  
\end{equation}
In particular, $\cH_K^0 \subeq \cH^\infty(\pi^H)$. 
For $x \in \fh$, the infinitesimal generator of the unitary 
one-parameter group $\pi^H_x(t) := \pi^H(\exp tx)$ is given by 
$\cL_x= \oline{\dd\pi^H}(x)$. Moreover, $\cH_K^0$, and hence $\cD$,
is a core for $\cL_x$.
\end{lem}

\subsection{The Integrability Theorem} 

\noindent The proof of the integrability theorem will be based on the following result.

\begin{thm} {\rm(\cite{Mer11})}\mlabel{T:intcrit}
Let $G^c$ be a simply
connected Banach--Lie group with Lie algebra $\g^c$. Assume 
that $\g^c=\fa_1\oplus\fa_2$ where $\fa_1$ and $\fa_2$ are closed subspaces. 
Let $\alpha$ be a 
strongly continuous representation of $\g$ on a dense subspace  $\cD$ 
of a Hilbert space $\cH$ such that 
for every $x\in\fa_1\cup\fa_2$, $\alpha(x)$ is essentially skew-adjoint, 
$e^{\oline{\alpha(x)}}\cD\subseteq\cD$ and 
\[ e^{\oline{\alpha(x)}}\alpha(y)e^{-\oline{\alpha(x)}}=\alpha(e^{\ad x}y)
\quad \mbox{ for } \quad y\in\g^c.\] 
Then $\alpha$ integrates to a continuous unitary representation 
$(\pi, \cH)$ of $G^c$ with $\cD \subeq \cH^\infty$ and 
$\alpha(x) = \dd\pi(x)|_{\cD}$ for $x \in \g^c$. 
\end{thm}

\begin{lem}
  \mlabel{lem:3.6}
For each $\phi \in \cD^1$, the complex linear map 
\[ \ev_\phi \: \g_\C \to \cH_K, \quad x \mapsto \cL_x \phi \] 
is continuous. In particular, the representation $\alpha:\g_\C\rightarrow\End(\cD)$ 
is strongly continuous.
\end{lem}

\begin{prf}
Since $\g_\C$ and $\cH$ are Banach spaces, 
it suffices to show that the graph of $\ev_\phi|_\fg$ is closed. 
This follows from the continuity of the linear functional 
\[ \g \to \C, \quad x \mapsto (\cL_x \phi)(m) 
= \dd\phi(m)\beta(x)(m)\] 
which follows from the smoothness of $\phi$ and 
our continuity assumption on $\beta$. 
\end{prf}

\begin{rem}
Let $m \in M$ and $x \in \g$ be such that $\Phi^{\beta(x)}_t(m)$ is defined. We 
have seen in Lemma~\ref{lem:vectrafo} that 
\[ \beta(e^{t\ad x}y)\res_{M_t} = (\Phi^{\beta(x)}_{-t})_*\beta(y) 
\quad \mbox{ for } \quad y \in \g.\] 
For the corresponding Lie derivatives, this means that 
\begin{equation}
  \label{eq:trarel}
(\cL_{\beta(e^{t\ad x}y)} \phi)(m) 
= \Big(\cL_{\beta(y)}(\phi \circ \Phi^{\beta(x)}_{-t})\Big)(\Phi^{\beta(x)}_t(m)) 
\quad \mbox{ for } \quad  m \in M_t, \phi \in C^\infty(M).
\end{equation} 
\end{rem}

\begin{thm} \mlabel{thm:4.8} Let $K$ be a smooth positive definite kernel 
on the manifold $M$ compatible with the smooth right action $(\beta,\sigma)$ of $(\g,H)$, 
where $\g = \fh \oplus \fq$ is a symmetric 
Banach--Lie algebra and $H$ is a connected Lie group with Lie algebra~$\fh$. 
Let $G^c$ be a simply connected Lie group with Lie algebra 
$\g^c = \fh + i\fq$. 
Then there exists a unique smooth unitary representation $(\pi^c,\cH_K)$ 
such that 
\begin{itemize}
\item[\rm(i)] $\oline{\dd\pi^c}(x) = \cL_x$ for $x \in \fh$. 
\item[\rm(ii)] $\oline{\dd\pi^c}(iy) = i\cL_y$ for $y \in \fq$. 
\end{itemize}
\end{thm}

\begin{prf} {\bf Step 1.} For $x \in \fq$, we consider the associated selfadjoint 
operator $\cL_x$ on $\cH_K$. Then we obtain a hermitian 
one-parameter group $e^{t\cL_x}$ of unbounded selfadjoint operators 
on $\cH_K$ defined by a spectral measure $P$ of $\cL_x$ via 
\[ e^{t\cL_x} := \int_\R e^{tx}\, dP(x) 
\quad \mbox{ on } \quad 
\cD(e^{t\cL_x}) = \Big\{ \phi \in \cH^K \: \int_\R e^{2tx} dP^\phi(x) < \infty 
\Big\},\] 
where $P^\phi(\cdot) = \la P(\cdot)\phi,\phi\ra$. 
In particular, we have, for each 
$t \in \R$, a well-defined domain 
\[ \cD_t= \cD(e^{t\cL_x}) \subeq \cH_K.\] 
Let $m \in M_s$. Then 
Theorem~\ref{T:Froelichgeo} implies that $K_m \in \cD_s$ with 
\[ e^{t\cL_x}K_m= K_{\Phi^{\beta(x)}_t(m)}\quad\text{for}\quad 0\leq t\leq s.\] 
The curve 
$t \mapsto e^{t\cL_x} K_m$ in $\cH_K$ extends to a holomorphic 
function $e^{z\cL_x} K_m$ defined on an open neighborhood of the strip 
$0 \leq \Re z \leq s$. Accordingly, for $y \in \g$, the function 
\[ t \mapsto \la e^{t\cL_x} K_m, \cL_y K_n \ra \] 
extends to a holomorphic function 
\[ z \mapsto \la e^{z\cL_x} K_m, \cL_y K_n \ra. \] 
For $x \in \fq$, $y \in \g$ and $m,n \in M_s$ we  
get:
\begin{align*}
\la e^{t\cL_x}K_m, \cL_{\tau(y)} K_n \ra 
&=  \la K_{\Phi^{\beta(x)}_t(m)}, \cL_{{\tau(y)}} K_n \ra 
=  \oline{ (\cL_{\tau(y)} K_n)(\Phi^{\beta(x)}_t(m))} \\
&=  \oline{ (\cL_{e^{t\ad x}{\tau(y)}}(K_n \circ \Phi^{\beta(x)}_t))(m)}\qquad \qquad 
\mbox{ by \eqref{eq:trarel}} \\
&
= \la \cL_{e^{t\ad x}{\tau(y)}}^* K_m, e^{t\cL_x}K_n \ra \\
&= -\la \cL_{e^{-t\ad x}y} K_m, e^{t\cL_x}K_n \ra 
\end{align*}
By analytic extension (cf.\ \cite[Lemma~2]{KL81}), 
we now arrive with Lemma~\ref{lem:3.6} at the relation 
\begin{align*}
\la e^{z\cL_x}K_m, \cL_{\tau(y)}  K_n \ra 
&= -\la \cL_{e^{-z\ad x}y} K_m, e^{\oline z\cL_x}K_n \ra 
\quad \mbox{ for } \quad 0 \le \Re z \leq s 
\end{align*}
and we get in particular 
\begin{align*}
\la e^{i\cL_x}K_m, \cL_{\tau(y)}  K_n \ra 
&= -\la \cL_{e^{-i\ad x}y}K_m, e^{-i\cL_x}K_n \ra 
\end{align*}
This relation leads to 
$e^{i\cL_x} K_m \in \cD_y$ with 
\[ \cL_y e^{i\cL_x} K_m = e^{i\cL_x} \cL_{e^{-i\ad x}y} K_m 
\quad \mbox{ for } \quad m \in M_s, y \in \g_\C.\] 
From $M = \bigcup_{s > 0} M_s$ we thus derive 
$e^{i\cL_x} \cH_K^0 \subeq \cD^1$  
with 
\[ \cL_y e^{i\cL_x} \res_{\cH_K^0} = e^{i\cL_x} \cL_{e^{-i\ad x}y} 
\res_{\cH_K^0},\] 
resp., 
\[e^{-i\cL_x} \cL_y e^{i\cL_x} \res_{\cH_K^0} = \cL_{e^{-i\ad x}y} 
\res_{\cH_K^0}\quad \mbox{ for } \quad x \in \fq, y \in \g_\C.\]

{\bf Step 2.} 
It follows from Step 1 that, for $m,n\in M$, $x\in\fq$ and $y\in\g^c$, we have
$$\la e^{i\cL_x}K_m, \cL_{y}  K_n \ra 
= \la K_m,\cL_{e^{-i\ad x}y} e^{-i\cL_x}K_n \ra $$
and by continuity
$$\la e^{i\cL_x}\varphi, \cL_{y}  K_n \ra 
= \la \varphi,\cL_{e^{-i\ad x}y} e^{-i\cL_x}K_n \ra=-\la \cL_{e^{-i\ad x}y}\varphi, e^{-i\cL_x}K_n \ra $$
for $\varphi\in\cD_{e^{-i\ad x}y}$. Hence $e^{i\cL_x}\varphi\in\cD((\cL_y|_{K_n})^*)$.
If $y\in \fh\cup i\fq$, then by Lemma~\ref{lem:3.1} and Corollary~\ref{C:esssa} we have 
$(\cL_y|_{\cH_K^0})^*=-\cL_y.$
It follows that
\[ e^{i\cL_x} \cD^1 \subeq \cD^1\quad\text{and}\quad e^{-i\cL_x} \cL_y e^{i\cL_x} \res_{\cD^1} = \cL_{e^{-i\ad x}y}|_{\cD^1}\quad \mbox{ for } \quad x \in \fq, y \in \g^c\]
and the commutation relation implies $e^{i\cL_x} \cD \subeq \cD.$

{\bf Step 3.}  For $x \in \fh$ we have 
\[ e^{\cL_x}\phi = \pi^H(\exp x)\phi = \phi \circ \sigma_{\exp x} 
\quad \mbox{ for } \quad \phi \in \cH_K.\] 
Therefore the relation 
\[ e^{-\cL_x} \cL_y e^{\cL_x} = \cL_{e^{-\ad x}y} \quad \mbox{ for } \quad 
x \in \fh, y \in \g^c \] 
follows from 
\[ (\sigma_{\exp x})_* \beta(y) = \beta(e^{-\ad x}y) \quad 
\mbox{ for } \quad x \in \fh, y \in \g\] 
(Remark~\ref{rem:3.2}). Since $\cD$ is dense by Proposition~\ref{lem:3.5}, we
 can now apply Theorem~\ref{T:intcrit} to complete the proof. 
\end{prf}

\begin{rem} Note that (i) implies that the restriction of 
$\pi^c$ to the integral subgroup $\la \exp \fh \ra \subeq G^c$ 
induces the same representation as $\pi^H$ on the universal 
covering group $\tilde H$ of~$H$. 
\end{rem}

\subsection{Examples} 

We now discuss a series of examples illustrating various situations in 
which the Integrability Theorem applies. 

\begin{exs} \mlabel{ex:3.15} (a) Let 
$(G,\tau)$ be a symmetric Banach--Lie group with Lie algebra 
$\g = \fh + \fq$, $H \subeq G$ the integral subgroup corresponding to the 
closed Lie subalgebra $\fh$, 
and $U = UH \subeq G$ be an open subset. We write $g^\sharp = \tau(g)^{-1}$ for 
$g \in G$. 
Then a function $\phi \: UU^\sharp\to \C$ is called 
{\it $\tau$-positive definite} if the kernel 
\[ K(x,y) := \phi(xy^\sharp) \] 
is positive definite. 

Then $\sigma_h(g) := gh$ and $\beta(x)(g) := g.x$ define a smooth 
right action of $(\g,H)$ on $U$ and the kernel $K$ is compatible 
with $(\beta,\sigma)$. Therefore we obtain for each simply connected 
Lie group $G^c$ a corresponding unitary representation $\pi^c$
on $\cH_K \subeq C^\infty(U,\C)$ with 
\[  (\pi^c(h)\psi)(g) = \psi(gh) \quad \mbox{ for } \quad 
g\in U, h \in H\]  
and 
\[  \dd\pi^c(iy)\psi = i\cL_y \psi \quad \mbox{ for } \quad 
\psi\in\cH_K^\infty, y \in \fq.\] 

(b) If $\rho \: G \to \GL(V)$ is a smooth action of the Banach--Lie group 
$G$ on the locally convex space and 
$( \cdot,\cdot)$ is a positive definite scalar product on $V$ 
satisfying 
\[ (\rho(g)v,w) = (v,\rho(g^\sharp)w) \quad \mbox{ for } \quad 
g \in G, v,w\in V,\] 
then $\rho$ defines in particular a smooth right action 
$(\beta,\sigma)$ on the manifold $M := V$ by 
\[ \sigma_h = \rho(h)^{-1} \quad \mbox{ and } \quad 
\beta(x)v = - \dd\rho(x)v,\] 
and the kernel $K(v,w) := (w,v)$ is compatible. 
Therefore we obtain for each simply connected 
Lie group $G^c$ a smooth unitary representation 
on $\cH_K$, which can be identified with the Hilbert completion 
of $V \subeq \cH_K^\infty$, and 
$\pi(h)\res_V = \rho(h)$ for $h \in H$. 

This situation occurs in particular for finite dimensional 
representations $(\rho, V)$ of real reductive Lie groups. 
In this case $\tau$ is a Cartan involution on $G$, 
a scalar product as above always exists, 
and the group $G^c$ is compact. Here the passage between unitary 
representations of $G^c$ and representations of $G$ is 
Weyl's Unitary Trick, used by H.~Weyl to study representations of 
$G$ by means of unitary representations of $G^c$. 
\end{exs}

\begin{ex} 
Let $V$ be a Banach space and $\Omega \subeq V$ be an open domain. 
Further, let $\mu$ be a positive measure on the smallest 
$\sigma$-algebra in the topological dual space $V'$ 
for which all evaluation maps $\ev_v(\alpha) := \alpha(v)$ are measurable. 
We assume that the Laplace transform 
\[ \cL(\mu)(x) := \int_{V'} e^{-\alpha(x)}\, d\mu(\alpha) \] 
defines a smooth function on $\Omega$. This is always the case 
if $\dim V < \infty$ and $\cL(\mu)(x) < \infty$ for $x \in \Omega$ 
(cf.\ \cite[Prop.~V.4.6]{Ne00}). Then 
\[ K(x,y) := \cL(\mu)\Big(\frac{x+y}{2}\Big) \] 
defines a smooth positive definite kernel on $\Omega$. 

We consider the symmetric Lie group 
$(G,\tau) = (V, -\id_V)$ for which $H = \{0\}$ and 
$\g = \fq = V$. For every $v \in V$, the corresponding 
constant vector field $\beta(v)(x) = v$ is locally integrable 
with maximal local flow given by $\Phi(t,x) = x + tv$ whenever 
$x+tv \in\Omega$. We thus obtain a smooth right action 
of $(\g,H)$ on $\Omega$. Since the kernel depends only on $x+y$, 
the vector fields $\beta(v)$ are $K$-symmetric. Therefore 
Theorem~\ref{thm:4.8} guarantees the existence of a corresponding 
unitary representation of the Banach--Lie group $G^c = iV$ on the corresponding 
reproducing kernel Hilbert space $\cH_K$. 

As the kernel $K$ extends holomorphically to the tube 
domain $T_\Omega = \Omega + i V$ by 
\[ K(z,w) := \cL(\mu)\Big(\frac{z+\oline w}{2}\Big), \] 
the corresponding Hilbert space can be identified with a space 
of holomorphic functions on $T_\Omega$ and the unitary 
representation of $G^c= iV$ is simply given by
$(\pi^c(iv)f)(z) := f(z+iv)$. 

One can show that $\cH_K \cong L^2(V',\mu)$ and that the elements 
of $\g = V$ act naturally by multiplication operators on $L^2(V',\mu)$. 
In particular, the operator $\cL^K_v$ is unbounded from above and below 
if the support of the measure $(\ev_v)_*\mu$ on $\R$ has this property. 
\end{ex}

\begin{ex} (a) Let $\cA$ be a real Banach-$*$-algebra and 
$G := \cA^\times$ be its group of units which is a Banach--Lie group 
whose Lie algebra $\g$ is $(\cA,[\cdot,\cdot])$. If $\cA$ is not unital, 
we embed it into $\cA_1 := \R \1\oplus \cA$ and define 
$\cA^\times := (\1 + \cA) \cap \cA_1^\times$. 
Then $G$ is a symmetric Banach--Lie group with respect to 
$\tau(a) := (a^*)^{-1}$ and $H = \{ a \in \cA^\times \: a^* = a^{-1}\}$ 
is the unitary/orthogonal group of $\cA$. 
If $\cA_\C$ is the complexification of $\cA$ and $(x+iy)^* := x^* - iy^*$ 
is the antilinear extension of $*$ to $\cA_\C$, then the unitary group 
$\U(\cA_\C)$ is a Banach--Lie group with Lie algebra 
\[ \g^c = 
\{ a \in \cA \: a^* = -a\} + 
i \{ b \in \cA \: b^* = b\} = \{ c \in \cA_\C \: c^* =  -c\}= \fu(\cA_\C).\]
Accordingly, we write $G^c$ for the simply connected covering group 
of the identity component $\U(\cA)_0$.

Let $S := \{ a \in \cA \: \|a\| < 1\}$ be the open unit ball 
in $\cA$ and observe that this is a $*$-semigroup. 
Let $\pi_S \: S \to B(\cH)$ be a $*$-representation of $S$ and 
$v \in \cH$ be a smooth vector for this representation, so that we obtain 
the positive definite smooth function 
$\phi(s) := \la \pi(s)v,v\ra$ and the corresponding smooth kernel 
$K(s,u) := \phi(su^*)$ on $S \times S$. 

Restricting the right multiplication action of $G = \cA^\times$ to 
$S$, we obtain a smooth right action of $(\g,H)$, for which 
the vector fields are given by $\beta(x)a = ax$. We then have 
\[ (\cL_x^1 K)(s,u) 
= \derat0 \phi(se^{tx}u^*)
= \derat0 \phi(s(ue^{tx^*}))
= (\cL_{x^*}^2 K)(s,u),\] 
so that the kernel $K$ is compatible with the smooth right action. 

We thus obtain from Theorem~\ref{thm:4.8} a smooth unitary representation of $G^c$ 
on the reproducing kernel space $\cH_\phi \subeq C^\infty(S)$ which 
can be identified with the cyclic subspace $\cH_v := \oline{\Spann \pi(S)v} 
\subeq \cH$. 

(b) If $\cA = M_n(\R)$ with $A^* = A^\top$ and 
$s \in \R$ is such that the kernel 
\[ K(x,y) 
:= \det(\1 - xy^\top)^{-s}  = \det(\1 - xy^*)^{-s} \] 
on the open unit ball $S$ is positive definite, then we likewise
 obtain a unitary representation of the group 
$G^c = \tilde\U_n(\C)$ on the corresponding reproducing kernel space. 

In this case the extension can also be obtained from the observation 
that the positive definiteness of the kernel $K$ implies the positive 
definiteness of the kernel 
\[ K_\C(x,y) := \det(\1 - xy^*)^{-s} \] 
on the open unit ball $S_\C \subeq M_n(\C)$  
(\cite[Thm.~A.1]{NO14}). 
\end{ex}

\begin{ex} (Application to operator-valued kernels) 

(a) Let $(G,\tau)$ be a symmetric Banach--Lie group 
and write $g^\sharp = \tau(g)^{-1}$. 
We consider a smooth right action of $G$ on the 
manifold $X$, a complex Hilbert space $V$, and a kernel $Q \: X \times X \to B(V)$. 
We assume that we have a function $J \: G \times X \to \GL(V), (g,x) \mapsto J_g(x)$ 
satisfying the cocycle condition 
\[ J_{g_1 g_2}(x) = J_{g_1}(x) J_{g_2}(x.g_1) \quad \mbox{ for } \quad g_1, g_2 \in G, x \in X, \] 
and that the kernel $Q$ satisfies the corresponding invariance relation 
\[ J_g(x) Q(x.g, y) = Q(x, y.g^\sharp) J_{g^\sharp}(y)^* 
\quad \mbox{ for } \quad x,y \in X, g \in G\]  
(cf.~\cite[Prop.~II.4.3]{Ne00}). 
On the set $M := X \times V$, we then obtain a $G$-right action by 
\[ (x,v).g := (x.g, J_g(x)^*v).\] 
We also obtain a positive definite kernel 
\[ K \: M \times M \to \C, \quad 
K((x,v), (y,w)) := \la Q(x,y)w, v \ra\]  
which satisfies the natural covariance condition 
\begin{align*}
K((x,v).g, (y,w)) 
&= K((x.g, J_g(x)^*v), (y,w)) 
= \la Q(x.g, y) w, J_g(x)^* v \ra\\ 
&= \la J_g(x) Q(x.g, y) w, v \ra 
= \la Q(x, y.g^\sharp) J_{g^\sharp}(y)^* w, v \ra \\
& = K((x,v), (y.g^\sharp, J_{g^\sharp}(y)^*w))
= K((x,v), (y,w).g^\sharp).
\end{align*}

Let $X_+ \subeq X$ be an open $H$-invariant subset on which the kernel 
$Q$ is positive definite, so that $K$ is positive definite on 
$M_+ := X_+ \times V$. 
The corresponding reproducing kernel Hilbert space 
$\cH_K \subeq \C^{M_+}$ consists of functions that are continuous and antilinear 
in the second argument, and it is easy to see that the map 
\[  \Gamma \: \cH_Q \to \cH_K,\quad 
\Gamma(f)(x,v) := \la f(x), v \ra \] 
is unitary. For the $G$-action $(g.f)(x) := J_g(x) f(x.g)$ on $V^X$, we have 
\[ \Gamma(g.f)(x,v) = \la J_g(x) f(x.g), v \ra 
= \la f(x.g), J_g(x)^* v \ra 
= \Gamma(f)((x,v).g), \] 
so that $\Gamma$ intertwines it with the action on $\C^M$ by 
\[ (g.F)(x,v) := F((x,v).g).\]

Let us assume that the $G$-action on $M_+$ is smooth, i.e., that 
$G \times X_+ \times V \to V, (g,x,v) \mapsto J_g(x)^*v$ is smooth. 
Then we obtain a smooth right action of $(\g,H)$ on $M_+$ compatible 
with the kernel $K$, and thus 
Theorem~\ref{thm:4.8} yields a unitary representation 
of $G^c$ on the Hilbert $\cH_K \cong \cH_Q$. 

(b) A concrete example of this type is obtained from the kernel 
$K(x,y) := e^{\la x,y \ra}$ on the infinite dimensional real Hilbert space $\cH$. 
Then $G := \GL(\cH)$ is a Banach--Lie group and 
$\tau(g) := (g^\top)^{-1}$ is an involutive automorphism satisfying 
\[ K(g.x, y) = K(x,g^\top.y) \quad \mbox{ for } \quad 
g \in G, x,y \in \cH.\] 
As the unitary group $\U(\cH_\C)$ is simply connected by Kuiper's Theorem, 
Theorem~\ref{thm:4.8} yields a unitary representation 
of $\U(\cH_\C)$ on the Fock space $\cH_K \subeq \C^\cH$ (which can also obtained 
more directly by analytic extension of the kernel). 
\end{ex}

\subsection{A generalization of the L\"uscher--Mack Theorem}

In this subsection we prove a generalization of the L\"uscher--Mack Theorem 
which applies to a class of semigroups that need not have a 
polar decomposition. 

\begin{thm}\mlabel{thm:4.11}  {\rm(Generalized L\"uscher--Mack Theorem)} 
Let $(G,\tau)$ be a connected symmetric Banach--Lie 
group. 
For $H := (G^\tau)_0$, we consider an open subsemigroup $S \subeq G$ with $SH=S$, and
which is invariant under the involution $s \mapsto s^\sharp := \tau(s)^{-1}$.   
Then, for every non-degenerate strongly continuous smooth representation\begin{footnote}{Smoothness of a representation means that the linear subspace 
$\cH^\infty$ of all vectors with smooth orbit map is dense in $\cH$.}
\end{footnote}
$\pi$ of the 
involutive semigroup $(S,\sharp)$, the following assertions hold: 
\begin{itemize}
\item[\rm(a)] There exist a smooth unitary representation $\pi^H$ of $H$ in $\cH$ and 
a Lie algebra representation $\alpha:\fg_\C\rightarrow \End(\cD)$ on the dense domain $\cD\subeq\cH$
such that $\alpha(x)$ is essentially skew-adjoint for $x\in \fh\cup i\fq$ and
\begin{itemize}
\item[\rm(i)] $\pi(sh)=\pi(s)\pi^{H}(h)$ for $s\in S$ and $h\in H$. 
\item[\rm(ii)] If $x\in\fq$ and $\exp tx\in S$ for $t>0$, then $\pi(\exp tx)=e^{t\overline{\alpha(x)}}$. 
\item[\rm(iii)] $\cD$ consists of smooth vectors for $\pi^H$ and $\alpha(x)v=\dd\pi^H(x)v$ for $v\in\cD$.
\end{itemize}
\item[\rm(b)] The representation $\alpha|_{\fg^c}$ integrates to a (smooth) unitary representation of any simply connected Lie group 
$G^c$ with Lie algebra $\g^c$.
\end{itemize} 
\end{thm}

\begin{prf} First we note that every non-degenerate smooth representation $(\pi, \cH)$ 
of $(S,\sharp)$ is a direct sum of cyclic ones generated by a smooth vector 
(see \cite[ Lemma II.4.2(iv)]{Ne00}). 

If the representation is generated by the smooth vector $v$, then the 
kernel $K(s,t) := \la \pi(st^*)v,v\ra$ on $S \times S$ 
is smooth and positive definite and satisfies 
\[ K(su, t) = K(s,tu^*) \quad \mbox{ for } \quad s,t,u \in S.\] 
Moreover, the unitary isomorhism
\[ \Gamma \: \cH \to \cH_K \subeq C^\infty(S,\C), \quad 
\Gamma(w)(s) := \la \pi(s)w, v \ra \] 
is an intertwining operator, where the representation 
$\cH_K$ is given by 
\[ (\pi_K(s)f)(t) := f(ts). \] 
Hence the assertion 
follows from Example~\ref{ex:3.15}(a). 
\end{prf}

\begin{rem} If $G$ is finite dimensional, then every non-degenerate 
continuous representation of $S$ is smooth. In fact, 
for every $\phi \in C^\infty_c(S)$, the image of the operators defined by 
\[ \pi(\phi)v = \int_S \phi(s)\pi (s)v\, d\mu_G(s) \] 
are smooth vectors. Here $\mu_G$ denotes a left Haar measure on $G$. 
\end{rem}

\section{Local Representations}
\mlabel{sec:4a} 

In this section we give two proofs of P.E.T.~Jorgensen's 
theorem on the analytic continuation of local
representations. The first one  only applies to finite
dimensional groups. As Jorgensen's original proof, it relies
on Simon's Exponentiation Theorem \cite{Si72}. However, thanks to a new idea, 
it is much shorter and does not use axioms (L5) and (L6) from \cite{Jo86}.
The second proof requires globality of the representation on the subgroup 
$H$ and a smoothness condition but has the advantage of applying to 
infinite dimensional Lie groups as well. The latter 
version can be derived from 
Theorem~\ref{thm:4.8}.

\subsection{Local representations of finite dimensional groups}

\begin{defn} Let $(G,H,\tau)$ be a symmetric Lie group and $\fh=\ker(\tau-\1)$, 
$\fq=\ker(\tau + \1)$ as above. 
Let $\cH$ be a Hilbert space and $U \subeq G$ be an identity neighborhood. 
Assume that, for every $g\in U$,  we are given a 
densely defined operator $\pi(g) \: \cD(\pi(g)) \to \cH$ on $\cH$ 
and that there exists a continuous unitary representation 
$(\pi_H,\cH)$ of $H$ such that $\pi_H(h)\res_{\cD(\pi(h))} = \pi(h)$ 
for every $h \in U \cap H$. 
Let $\cD$ be a dense subspace in $\cH$. Then $(\pi,\cD,\cH)$
is called a {\it local representation of $(G,H,\tau)$} if
the following assumptions are satisfied.
\edz{added (LO)}\begin{itemize}
\item[(L0)] If $g,g^\sharp\in U$ then
$\pi (g^\sharp )\subeq \pi (g )^*$.
 \item[(L1)] $\cD\subeq\cD(\pi(g))$ for every $g\in U$.
 \item[(L2)] If $g_1$, $g_2$ and $g_1g_2$ are in $U$, then
 $$\pi(g_2)\cD\subeq \cD(\pi(g_1))\quad\text{and}\quad \pi(g_1g_2)\res_{\cD}
=\pi(g_1)\pi(g_2)\res_{\cD}.$$
 \item[(L3)] There exists a starlike $0$-neighborhood $U_\fq$
 in $\fq$ such that $\exp U_\fq\subeq U$ and for every $v\in\cD$
 $$\lim_{t\rightarrow 0}{\pi(\exp ty)}v=v.$$
 \end{itemize}
 
 It follows from (L1-3) and \cite[Theorem I.1]{Fro80} that, for $y\in\fq$,
 the local semigroup $\pi(\exp(ty))$ has a unique extension as a local selfadjoint
 semigroup $e^{tB_y}$, where $B_y$ is a selfadjoint operator with $\cD\subeq \cD(B_y)$. We will
 moreover assume:
\begin{itemize}
 \item[(L4)] If $y\in\fq$, then $B_y\cD\subeq\cD$. 
 \end{itemize}
\end{defn}

For $x\in\fh$, we write $B_x$ for the skew-adjoint generator of $\pi_H(\exp tx)$. 

\begin{thm}
Let $(\pi,\cD,\cH)$ be a local representation of the finite dimensional Lie group $G$. 
Let $\fg^c=\fh+ i\fq$ be 
the $c$-dual Lie algebra and $G^c$ the corresponding simply connected
Lie group. Then there exists a unitary representation $(\pi^c,\cH)$ of $G^c$ such that
\begin{enumerate}
 \item[\rm(i)] $\overline{\mathrm{d}\pi^c(x)}=B_x$ for $x\in\fh$,
 \item[\rm(ii)] $\overline{\mathrm{d}\pi^c(iy)}=iB_y$ for $y\in\fq$
\end{enumerate}
\end{thm}

\begin{proof}{\bf Step 1.} The first step of the proof follows \cite[Observation 1-3]{Jo86}. For
$x\in\fh_1:=[\fq,\fq]$ one shows that $\cD\subeq \cD(B_x)$ and $B_x\cD\subeq \cD$.
It follows that one can define a representation $\alpha$ of $\g_1^c:=\fh_1+ i\fq$ on
$\cD$ by skew-symmetric operators via the formula
$$\alpha(x+iy):=B_x|_\cD+i B_y|_\cD.$$
Assumption (L4) implies with \cite[Lemma A]{Jo87}  that $\cD$ consists of analytic vectors
for every $y\in\fq$. Since, by construction, the Lie algebra $\fg^c_1$ is generated
by $\fq$, Simon's Exponentiation Theorem \cite{Si72} applies
and $\alpha$ integrates to a unitary representation $(\pi^c_1,\cH)$ of the simply
connected group $G_1^c$ with Lie algebra $\fg_1^c$.

{\bf Step 2.} We consider the semidirect product $H\ltimes G_1^c$ where the
action of $H$ on $G_1^c$ is obtained by integrating the adjoint action. We claim that
the unitary representations $\pi_H$ and $\pi_1^c$ combine to a unitary representation
$\hat{\pi}$ of $H\ltimes G_1^c$. It suffices to show that for $h\in H$ and $y\in\fq$
the commutation relation
\begin{equation}\label{E:comrel3}
\pi_H(h)\pi_1^c(\exp iy)=\pi_1^c(\exp i\Ad(h)y)\pi_H(h)
\end{equation}
holds. Let $s>0$ such that $s y$ and 
$s \Ad(h)y$ are in $U_\fq$. Then, for $\Re(z)<s$
and $v,w\in\cD$, the maps $F(z)=\langle\pi(h)e^{zB_y}v,w\rangle$
and $G(z)=\langle e^{zB_{\Ad(h)y}}\pi(h)v,w\rangle$ are analytic. Since
$F(t)=G(t)$ for $t\in(0,s)$ by (L2), we have by analytic continuation
$F(i)=G(i)$ and \eqref{E:comrel3} follows. The kernel of the 
representation $\hat{\pi}$ contains the subgroup 
$\Delta := \{(h,h^{-1})\mid h\in h\}\subeq H\ltimes G_1^c$, so that it factors through 
a representation $\pi^c$ of $G^c \cong (H \ltimes G_1^c)/\Delta$. 
\end{proof}

\subsection{Local representations of infinite dimensional groups}

Since a (continuous) unitary representation of an infinite dimensional
Lie group does not necessarily possesses smooth vectors (\cite{Ne10}), we need to add
the following assumption in the definition of a local representation for
an infinite dimensional Lie group.
\begin{enumerate}
 \item[(LI1)] There exists a possibly finer topology on $\cD$ such that the action map
 $U\times \cD\rightarrow \cH,\  (g,v)\mapsto \pi(g)v$
 is smooth. 
\end{enumerate}
We then obtain an unbounded operator $\dd\pi(x):\cD \rightarrow \cH$ by the formula
$$\dd\pi(x)v:=\frac{\dd}{\dd t}\Bigr|_{t=0}\pi(\exp tx)v,\ v\in\cD$$
To be able to use the results of Section 3-4, we also add the following assumption:
\begin{enumerate}
 \item[(LI2)] $UH=U$.
\end{enumerate}

It turns out that with these new assumptions, condition $\mathrm{(L4)}$ is not needed
anymore.
\begin{thm}
Let $(\pi,\cD,\cH)$ be a local representation of the Banach--Lie group $G$ satisfying {\rm (L1-3)}
and 
$\mathrm{(LI1/2)}$. Let $\fg^c=\fh+ i\fq$ be 
the $c$-dual Lie algebra and $G^c$ be a simply connected
Lie group with Lie algebra $G^c$. Then there exists a unitary representation $(\pi^c,\cH)$ of $G^c$ such that
$$\dd\pi^c(x+iy)|_\cD=\dd\pi(x)+i\dd\pi(y)\quad\text{for}\quad x\in\fh,y\in\fq.$$
\end{thm}
\begin{prf}
Let us define $\gamma(g,v):=\pi(g^\sharp)v$ where $g^\sharp:=\tau(g)^{-1}$.
Then 
$$K((g,v),(k,w))):=\la\gamma(k,w),\gamma(g,v)\ra$$ is
a smooth positive definite kernel on $U\times \cD$ and $\Gamma:\cH\rightarrow \cH_K\subeq C^\infty(U\times \cD,\C)$,
$$\Gamma(w):=\la w,\gamma(g,v)\ra$$
is a Hilbert space isomorphism. Let $(\beta,\sigma)$ be the smooth right action of $(\fg,H)$ on $U\times \cD$ defined by
$$\beta(x)(g,v):=(g.x,0),\quad \sigma(h)(g,v)=(gh,v),\quad x\in\g, h\in H, g\in U, v\in \cD.$$
For $h\in H$ we have $h^\sharp=h^{-1}$ and we thus get 
$$K((g,v)h,(k,w))=K((g,v),(k,w)h^{-1}).$$
If $x\in\fg$, then $\Phi^{\beta(x)}_t(g,v)=(g\exp tx,v)$ and
\begin{align*}
\cL_{\beta(x)}^1K((g,v),(k,w))&=\frac{\dd}{\dd t}\Bigr|_{t=0}\la \pi(k^\sharp)w,\pi((g\exp tx)^\sharp)v\ra\\
&=\frac{\dd}{\dd t}\Bigr|_{t=0}\la \pi((k\exp -t\tau(x))^\sharp)w,\pi(g^\sharp)v\ra\\
&=-\cL^2_{\beta{(\tau(x))}}K((g,v),(k,w)).
\end{align*}
Hence $K$ is $\beta$-compatible and by Theorem~\ref{thm:4.8}, $\cH_K$ carries on 
$$\cD_K:=\{\phi\in\cH_K\mid (\forall n\in\N_0, x_1, \ldots, x_n \in\g)\ 
\cL_{\beta(x_n)}\dots \cL_{\beta(x_1)}\phi\in\cH_K\}$$
a representation $\alpha$ of $\g_\C$ such that $\alpha|_{\g^c}$ integrates to a unitary representation of
$G^c$. To complete the proof, it remains to observe that $\cD\subeq\Gamma^{-1}(\cD_K)$ and
$\Gamma(\dd\pi(x)w)=\cL_{\beta(x)}\Gamma(w)$ for $x\in\g$. 
\end{prf}

\section{Generalization to distribution kernels} 
\mlabel{sec:4}

In this section we consider the case where $M$ is a finite dimensional manifold and
$K\in C^{-\infty}(M\times M)$ is a positive definite distribution compatible 
with a smooth right action $(\beta,\sigma)$ of $(\g,H)$. Our main Integrability theorem
is Theorem \ref{thm:4.12}.

\subsection{Geometric Fr\"ohlich Theorem for distributions}

\begin{defn}
Let $M$ be a finite dimensional smooth manifold and 
$K \in C^{-\infty}(M \times M)$ be a distribution, i.e., an antilinear 
continuous functional on the LF space $C^\infty_c(M) \to \C$. 
We say that $K$ is {\it positive definite} if 
\[ \la \phi, \psi \ra_K := K(\oline\phi \otimes \psi) \] 
defines a positive semidefinite hermitian form on $C^\infty_c(M)$. 
Then we write $\cH_K$ for the corresponding Hilbert space completion 
and $K_\phi$ for the image of $\phi \in C^\infty_c(M)$ in  $\cH_K$. 
Since the canonical map 
$\iota \: C^\infty_c(M) \to \cH_K, \phi \mapsto K_\phi$ 
is continuous linear with dense range,\begin{footnote}
  {To see that $\iota$ is continuous, it suffices to verify this claim 
on the subspace $\cD_C(M)$ for a compact subset $C \subeq M$. 
On this Fr\'echet space, the map 
$\cD_C(M) \to \cD_{C \times C}(M \times M), \phi \mapsto 
\phi \otimes \phi$ is easily seen to be continuous. 
Therefore the continuity of $K$ implies that 
$\phi  \mapsto \la \phi, \phi\ra_K$ is bounded on some $0$-neighborhood, 
and this means that $\iota\res_{\cD_C(M)}$ is continuous. }
\end{footnote}
we have an injective adjoint map 
\[ \iota' \: \cH_K \to C^{-\infty}(M), \quad 
\iota'(v)(\phi) := \la v, K_\phi\ra.\] 
In the following we accordingly identify $\cH_K$ with the  subspace 
$\iota'(\cH_K)$ of $C^{-\infty}(M)$. 
We then have the relation
\[ K_\varphi(\psi)=\la K_\phi, K_\psi \ra = \la \phi, \psi \ra_K 
= K(\oline \phi \otimes \psi) \quad \mbox{ for } \quad 
\phi, \psi \in C^\infty_c(M).\] 
\end{defn}

The continuity of the linear map $\iota$ implies that the kernel 
$K$ is a continuous bilinear map, hence a smooth kernel on the linear locally 
convex manifold $C^\infty_c(M)$. Therefore the results of 
Section~\ref{sec:2b} apply in particular to $K$ as a smooth kernel on $C^\infty_c(M)$. 

\begin{defn} The Lie derivative defines on  $C^\infty_c(M)$ the structure of a 
$\cV(M)$-module, and we consider on $C^{-\infty}(M)$ the adjoint representation: 
\[ (\cL_X D)(\phi) := - D(\cL_X \phi) \quad \mbox{ for } \quad X \in \cV(M), D \in C^{-\infty}(M), 
\phi \in C^\infty_c(M).\] 

For a distribution $D \in C^{-\infty}(M \times M)$
and $X \in \cV(M)$, we write 
\[ (\cL_X^1 D)(\phi \otimes \psi) := 
- D(\cL_X \phi \otimes \psi) \quad \mbox{ and } \quad 
(\cL_X^2 D)(\phi \otimes \psi) := 
- D( \phi \otimes \cL_X\psi),\] 
and we say that $X$ is {\it $D$-symmetric} if $\cL^1_XD=\cL^2_XD$ and
{\it $D$-skew-symmetric} if $\cL^1_XD=-\cL^2_XD$.
\end{defn}

\begin{rem}\label{R:sym} Let $K$ be a positive definite distribution on $M$.
If $X$ is $K$-symmetric (resp. $K$-skew-symmetric), then $\cL_X$ defines 
a symmetric (resp. skew-symmetric) operator on $C^\infty_c(M)$ 
with respect to $\la\cdot,\cdot\ra_K$.
\end{rem}

The next observation allows us to use Corollary~\ref{T:3.7} and to adapt the methods 
used in Section~\ref{sec:3}. 
\begin{rem} Let $X\in\cV(M)$ be locally integrable and
$\varphi\in C_c^\infty(M)$. If $\supp\varphi\subeq M_{-t}$,
then $\varphi\circ \Phi_t^X$ has compact support $\Phi_{-t}^X(\supp\varphi)\subeq M_{t}$
and therefore can be seen as an element of $C_c^\infty(M)$. 
\end{rem}

\begin{thm} {\rm(Geometric Fr\"ohlich Theorem for distributions)} 
\mlabel{T:Froelich-dist}
Let $M$ be a smooth manifold and $K \in C^{-\infty}(M \times M)$ be a positive definite 
distribution. If 
$X \in \cV(M)$ is a locally integrable $K$-symmetric  vector field on $M$, 
then the Lie derivative $\cL_X$ defines an essentially selfadjoint 
operator $\cH_K^0 \to \cH_K$ whose closure 
$\cL_X^K$ coincides with $\cL_X|_{\cD_X}$, where 
\[ \cD_X := \{D \in \cH_K \: \cL_XD \in \cH_K\}. \] 
Moreover, if the local flow $\Phi^X$ is defined on 
$[-\varepsilon,0] \times \supp(\phi)$ for some $\phi \in C^\infty_c(M)$, then 
\[ e^{t\cL_X^K} K_\phi = K_{\phi \circ \Phi^X_{-t}}\quad\text{for}\quad 0\leq t\leq \varepsilon\,.\] 
\end{thm}

\begin{prf} For every $\varphi \in C^\infty_c(M)$, 
there exists an $\eps > 0$ such that the flow $\Phi^X$ of $X$ is defined on 
the compact subset $[-\eps, 0] \times \supp(\varphi)$ of 
$\R \times M$. Then the curve 
\[ \gamma \: [0,\eps] \to C^\infty_c(M), \quad \gamma(t) 
:= \varphi \circ \Phi_{-t}^X \] 
satisfies $\gamma'(t)=-\cL_X\varphi$ in the natural topological on~$C^\infty_c(M)$. 
Therefore the assumptions of Corollary~\ref{T:3.7} are satisfied with $V=C_c^\infty(M)$ and $L=\cL_X$. 
We conclude that $\cL_X|_{\cH_K^0}$ is essentially self-adjoint with closure 
equal to $\cL_X^K$ and that $e^{t\cL_X^K} K_\phi = K_{\phi \circ \Phi^X_{-t}}$ for 
$0\leq t\leq \eps$. 
\end{prf}

\subsection{The Integrability Theorem for Distributions}

\begin{defn} 
Let $\g = \fh + \fq$ be a symmetric Lie algebra with involution 
$\tau$ and $\beta \: \g \to \cV(M)$ be a homomorphism of Lie algebras. 
A positive definite distribution 
$K \in C^{-\infty}(M \times M,\C)$ is said to be 
{\it $\beta$-compatible} if 
\[ \cL^1_{\beta(x)} K = -\cL^2_{\beta(\tau(x))} K \quad \mbox{ for } \quad 
x \in \g.\] 
\end{defn}

In the following we assume that $K$ is a positive definite 
distribution on $M$ compatible with the smooth right 
action $(\beta,\sigma)$ of $(\g,H)$ (cf.\ Definition~\ref{def:3.1}).
For $z\in\g_\C$, we put 
\[ \cL_{\beta(z)}:=\cL_{\beta(x)}+i\cL_{\beta(y)} \] 
and we write $\cL_z$ for the restriction of $\cL_{\beta(z)}$ to
its maximal domain 
$$\cD_z=\{D\in\cH_K\mid \cL_{\beta(z)}D\in\cH_K\}.$$
As in Section~\ref{sec:2b}, we define $\cD_1:=\bigcap_{x\in\fg}\cD_x$ and
\[ \cD := \{ D \in \cH_K \: (\forall n \in \N)
(\forall x_1,\ldots, x_n \in \g)\, \cL_{\beta(x_1)}\cdots 
\cL_{\beta(x_n)}D \in \cH_K\} \]  
so that we obtain a Lie algebra representation 
$\alpha \: \g_\C \to \End(\cD)$ 
such that, for $x \in \g^c := \fh + i \fq$, the operator 
$\alpha(x)$ is skew-hermitian. 
From $\eqref{E:8}$ and Remark~\ref{R:sym} we deduce that 
\begin{equation}\label{E:infaction}
\cL_x K_\varphi=K_{\cL_{\tau(x)}\varphi},
\end{equation}
hence
$\cH_K^0 \subeq \cD$. In particular, $\cD$ is dense in $\cH_K$.

\begin{lem}
  \mlabel{lem:4.5} 
The prescription 
$\pi^H(D)(\phi) := D(\phi \circ \sigma_h^{-1})$ 
defines a smooth unitary representation $(\pi^H, \cH_K)$ of $H$.
For $h \in H$ and $\phi \in C^\infty_c(M) $, we have
\begin{equation}
  \label{eq:disinv2}
 \pi^H(h)K_\phi := K_{\phi \circ \sigma_h} 
\end{equation}
and $\cH_K^0$ consists of smooth vectors. 
For $x \in \fh$, the infinitesimal generator 
$\oline{\dd\pi^H}(x)$ of the unitary 
one-parameter group $\pi^H_x(t) := \pi^H(\exp tx)$
coincides with $\cL_x$. Moreover $\cH_K^0$ is a
core for $\cL_x$.
\end{lem}

\begin{prf} The $H$-invariance of $K$ implies that 
the hermitian form $\la \cdot,\cdot\ra_K$ on 
$C^\infty_c(M)$ is invariant under the action of $H$ on $C^\infty_c(M)$. 
This implies that the subspace $\cH_K \subeq C^{-\infty}(M)$ 
is invariant under the natural $H$-action on $C^{-\infty}(M)$ given by 
$(h.D)(\phi) = D(\phi \circ \sigma_h^{-1})$ and that it restricts 
to a unitary representation 
(cf.\ \cite[Prop.~II.4.9]{Ne00}). 
From \eqref{E:infaction} and the connectedness of $H$, we derive that 
\[ h.K_\phi = K_{\phi \circ \sigma_h} \quad \mbox{ for } \quad h \in H.\] 

The continuity of $\iota$ and the smoothness of the 
action of $H$ on $C^\infty_c(M)$ further imply that $\cH_K^0$ consists of smooth 
vectors. The second part of the lemma follows from Lemma~\ref{lem:3.2}.  
\end{prf}

\begin{lem}
  \mlabel{lem:3.6b}
For each $D \in \cD^1$, the complex linear map 
\[ \omega_D \: \g_\C \to \cH_K, \quad x \mapsto \cL_x D \] 
is continuous. 
\end{lem}

\begin{prf}
Since $\g_\C$ and $\cH$ are Banach spaces, 
it suffices to show that the graph of $\omega_D$ is closed. 
This follows from the continuity of the linear functionals 
\[ \g_\C \to \C, \quad x \mapsto (\cL_x D)(\phi) 
= - D(\cL_x \phi), \quad \phi \in C^\infty_c(M).\qedhere\]
\end{prf}

To show that the representation $\alpha \: \g^c \to \End(\cD)$ integrates to a 
continuous unitary representation of $G^c$ we will again use
Theorem~\ref{T:intcrit}. 

\begin{thm} \mlabel{thm:4.12} 
Let $K \in C^{-\infty}(M \times M)$ be a positive definite distribution 
compatible with the smooth right action 
$(\beta,\sigma)$ of the pair $(\g,H)$ on $M$, where 
$\g = \fh \oplus \fq$ is a symmetric 
Banach--Lie algebra and $H$ is a connected Lie group with Lie algebra $\fh$. 
Let $G^c$ be a simply connected Lie group with Lie algebra 
$\g^c = \fh + i\fq$. 
Then there exists a unique smooth unitary representation $(\pi^c,\cH_K)$ of $G^c$ 
such that 
\begin{description}
\item[\rm(i)] $\oline{\dd\pi^c}(x) = \cL_x$ for $x \in \fh$. 
\item[\rm(ii)] $\oline{\dd\pi^c}(iy) = i\cL_y$ for $y \in \fq$. 
\end{description}
\end{thm}

\begin{prf} The proof is very similar to the one
of Theorem~\ref{thm:4.8}. The first step has to be slightly adapted.
For $x \in \fq$, we consider the associated selfadjoint 
operator $\cL_x$ on $\cH_K$ (Theorem~\ref{T:Froelichgeo}). 
Then we obtain by spectral calculus a hermitian 
one-parameter group $e^{t\cL_x}$ of unbounded selfadjoint operators 
on $\cH_K$ and we put $\cD_t= \cD(e^{t\cL_x})$. 
For $\phi \in C^\infty_c(M)$ with $\supp(\phi) \subeq M_{s}$ we know 
from Theorem~\ref{T:Froelich-dist} that $K_\phi \in \cD_s$ with 
\[ \gamma(t) = K_{\phi \circ \Phi^{\beta(x)}_{-t}} = e^{t\cL_x}K_\phi\quad\text{for}\quad 0\leq t\leq s.\] 
The curve 
$t \mapsto e^{t\cL_x} K_\phi$ in $\cH_K$ extends to a holomorphic 
function $e^{z\cL_x} K_\phi$ defined on an open neighborhood of the strip 
$0 \leq \Re z \leq s$. Accordingly, for $y \in \g$ and 
$\psi \in C^\infty_c(M)$, the function 
\[ t \mapsto \la e^{t\cL_x} K_\phi, \cL_y K_\psi \ra \] 
extends to a holomorphic function 
\[ z \mapsto \la e^{z\cL_x} K_\phi, \cL_y K_\psi \ra. \] 
For $x \in \fq$, $y \in \g$ and $\phi, \psi \in C^\infty_c(M)$ with 
$\supp(\phi), \supp(\psi) \subeq M_{s}$, we thus get:  
\begin{align*}
\la e^{t\cL_x}K_\phi, \cL_{\tau(y)} K_\psi \ra 
&=  \la K_{\phi \circ \Phi^{\beta(x)}_{-t}}, \cL_{{\tau(y)}} K_\psi \ra 
=  \oline{ (\cL_{\tau(y)} K_\psi)(\phi \circ \Phi^{\beta(x)}_{-t})} \\
&=  \oline{ (\cL_{e^{t\ad x}\tau(y)}K_{\psi \circ \Phi^{\beta(x)}_{-t}})}(\phi)
\qquad\qquad \mbox{ by  \eqref{eq:trarel}} \\
&
= \la  K_\phi, \cL_{e^{t\ad x}{\tau(y)}} e^{t\cL_x}K_\psi \ra \\
&= -\la \cL_{e^{-t\ad x}y} K_\phi, e^{t\cL_x}K_\psi \ra 
\end{align*}
By analytic extension (cf.\ \cite[Lemma~2]{KL81}), 
we now arrive with Lemma~\ref{lem:3.6b} at the relation 
\begin{align*}
\la e^{z\cL_x}K_\phi , \cL_{\tau(y)}  K_\psi \ra 
&= -\la \cL_{e^{-z\ad x}y} K_\phi, e^{\oline z\cL_x}K_\psi \ra 
\quad \mbox{ for } \quad 0 \le \Re z \leq t 
\end{align*}
and we get in particular 
\begin{align*}
\la e^{i\cL_x}K_\phi, \cL_{\tau(y)}  K_\psi \ra 
&= -\la \cL_{e^{-i\ad x}y}K_\phi, e^{-i\cL_x}K_\psi \ra.  
\end{align*}
This relation leads to 
$e^{i\cL_x} K_\phi \in \cD_y$ with 
\[ \cL_y e^{i\cL_x} K_\phi = e^{i\cL_x} \cL_{e^{-i\ad x}y} K_\phi 
\quad \mbox{ for } \quad \phi \in C^\infty_c(M), y \in \g.\] 
Here we use that $M = \bigcup_{s > 0} M_s$ implies that 
$\supp(\phi) \subeq M_s$ for some $s > 0$. 
We thus obtain $e^{i\cL_x} \cH_K^0 \subeq \cD^1$  with 
\[ \cL_y e^{i\cL_x} \res_{\cH_K^0} = e^{i\cL_x} \cL_{e^{-i\ad x}y} 
\res_{\cH_K^0},\] 
resp., 
\[e^{-i\cL_x} \cL_y e^{i\cL_x} \res_{\cH_K^0} = \cL_{e^{-i\ad x}y} 
\res_{\cH_K^0}\quad \mbox{ for } \quad x \in \fq, y \in \g^c.\] 
Now the rest of the proof is exactly as in Steps 2 and 3 of the proof of Theorem~\ref{thm:4.8}.
\end{prf}

\begin{rem} Let $(\beta,\sigma)$ be a smooth right action of 
$(\g,H)$ on $M$. Then 
$V := C^\infty_c(M)$ is a locally convex manifold on which the Lie derivatives 
$\cL_{\beta(x)}$ defines a representation 
$\g \to \cV(V)$. Moreover, every positive definite distribution $K$ on 
$M \times M$ defines a smooth positive definite kernel 
$\tilde K(\phi, \psi) := K(\oline\phi \otimes \psi)$ on $V$. 
One is therefore tempted to derive the results in Section~\ref{sec:4} 
from Section~\ref{sec:3}. However, this does not work directly because 
the linear vector fields on $V$ defined by the Lie derivatives are not locally 
integrable if the corresponding flow on $M$ is not global 
(cf.\ Proposition~\ref{prop:lin}). 
\end{rem}

\subsection{Reflection positive distributions and representations}

In this subsection we connect the previously obtained integrability result 
to reflection positivity (cf.\ \cite{NO13, NO14, JOl98, JOl00}). 
Let 
$D \in C^\infty(M \times M,\C)$ be a 
reflection positive distribution kernel which is reflection positive w.r.t.\ 
the involution $\theta \: M\to M$ on the open subset $M_+ \subeq M$ 
(cf.\ Definition~\ref{ex:1.4}). 
Our main result is Theorem~\ref{thm:4.12b} which shows that, 
under the natural compatibility requirements for an action of a 
symmetric Lie group $(G,H,\tau)$ on $(M,\theta)$, the representation 
of the pair $(\g^c,H)$ on the Hilbert space $\cH_{D_+}$ corresponding to the 
positive definite distribution $D(\cdot, \theta \cdot)$ on $M_+$ integrates to a 
unitary representation of the corresponding simply connected group $G^c$ with 
Lie algebra~$\g^c$.
 
\begin{defn} \mlabel{def:x.1} Let $\cE$ be a Hilbert space and 
$\theta \in \U(\cE)$ be an involution. 
We call a closed subspace $\cE_+ \subeq \cE$ {\it $\theta$-positive} 
if $\la \theta v,v \ra \geq 0$ for $v \in \cE_+$. 
We then say that  the triple $(\cE,\cE_+,\theta)$ is a {\it reflection positive 
Hilbert space}. In this case we write 
\[\cN 
:= \{ v \in \cE_+ \: \la \theta v, v \ra = 0\} 
= \{ v \in \cE_+ \: (\forall w \in \cE_+)\ \la \theta w, v \ra = 0\} 
= \cE_+ \cap \theta(\cE_+)^\bot, \] 
$q \: \cE_+ \to \cE_+/\cN, v \mapsto \hat v = q(v)$ for the quotient map 
and $\hat\cE$ for the Hilbert completion of $\cE_+/\cN$ with respect to 
the norm $\|\hat v\|_{\hat{\cE}}:=\|\hat v\| := \sqrt{\la \theta v, v \ra}$. 
\end{defn}

Let $(G,H,\tau)$ be a symmetric Lie group and $\g = \fh \oplus \fq$
be the corresponding symmetric Lie algebra. 
\begin{defn} Let  $(G,H,\tau)$ be a symmetric Lie group and 
$(\cE,\cE_+,\theta)$ be a reflection positive Hilbert space. 
A  unitary representation $\pi \: G \to \cE$ is said to be {\it reflection 
positive} on $(\cE,\cE_+,\theta)$ if the following 
three conditions hold:
\begin{description}
\item[\rm(RP1)] $\pi(\tau(g)) = \theta \pi(g)\theta$ for every $g \in G$. 
\item[\rm(RP2)] $\pi (h)\cE_+= \cE_+$ for every $h \in H$. 
\item[\rm(RP3)] There exists a subspace $\cD\subeq \cE_+\cap \cE^\infty$, dense in $\cE_+$, 
such that  $\dd \pi (X)\cD\subset \cD$ for all $X\in \fq$. 
\end{description}
\end{defn}


\noindent If $\pi$ is a reflection positive representation on $(\cE,\cE_+,\theta)$, 
then it follows from \cite[Lemma~2.4]{NO13} that
by defining $\hat\pi_0(h)=\hat{\pi (h)}$ we get a  unitary representation 
$(\hat\pi_0, \hat\cE)$ of $H$. However, we would like to have a 
unitary representation $\pi^c$ of $G^c$ on $\hat\cE$ 
``extending'' $\hat\pi_0$. In the following we give a geometric construction
of reflection positive representations and we show,
using the previous results, that they can be analytically 
continued to unitary representations of $G^c$.

\begin{defn} \mlabel{ex:1.4} Let $M$ be a smooth finite dimensional manifold and 
$D \in C^{-\infty}(M \times M)$ be a positive definite distribution. 
Suppose further that $\theta \: M \to M$ is an involutive diffeomorphism and 
that $M_+ \subeq M$ is an open subset such that the distribution $D_+$ 
on $M_+ \times M_+$ defined by 
\[ D_+(\phi) := D(\phi \circ (\theta \times \id_M)) \] 
is positive definite. We then say that $D$ is {\it reflection positive} with respect to 
$(M,M_+, \theta)$. 
\end{defn}

Let $D \in C^{-\infty}(M \times M)$ be a reflection positive distribution with respect to 
$(M,M_+, \theta)$ as in Example~\ref{ex:1.4} and let
$\cE\subeq C^{-\infty}(M)$ be the corresponding reflection positive 
Hilbert space obtained by completing $C^\infty_c(M)$ with respect to the scalar product 
\[ \la \phi, \psi \ra_D := D(\oline \phi \otimes \psi).\] 
Then the closed subspace $\cE_+$ generated by $C^\infty_c(M_+)$ is 
$\theta$-positive with respect to $\theta \phi := \phi \circ \theta$ 
and $\hat\cE \cong \cH_{D_+} \subeq \di(M_+)$ 
(cf.\ \cite{NO13}).

Let $(G,H,\tau )$ be a symmetric Lie group acting on $M$ such that
$\theta (g.m)=\tau(g).\theta (m)$ and $H.M_+=M_+$. We 
assume   that $D$ is invariant under $G$ and $\tau$. 
Then we have a natural unitary representation $(\pi_\cE, \cE)$ of $G$. 
As $M_+$, and therefore $\cE_+$, is $H$-invariant, this representation is reflection 
positive.  

From the invariance condition 
\begin{equation}
  \label{eq:invcon2b}
\cL_{\beta(x)}^1 D = - \cL_{\beta(x)}^2 D \quad \mbox{ for} \quad 
x \in \g, 
\end{equation}
we derive 
\begin{eqnarray} 
  \label{eq:invcon3}
\cL_{\beta(x)}^1 D_+ = - \cL_{\beta(\tau(x))}^2 D_+ \quad \mbox{ for} \quad 
x \in \g.\end{eqnarray} 

This implies that the assumptions of Theorem~\ref{thm:4.12} 
are satisfied, so that we obtain: 

\begin{theorem}\mlabel{thm:4.12b} Let $M$ be a smooth finite dimensional manifold and 
$D \in C^{-\infty}(M \times M)$ be a positive definite distribution which is 
reflection positive w.r.t.\ $(M,M_+, \theta)$. 
Let $(G,H,\tau )$ be a symmetric Lie group acting on $M$ such that
$\theta (g.m)=\tau(g).\theta (m)$ and $H.M_+=M_+$. We 
assume   that $D$ is invariant under $G$ and $\tau$. 
Let $G^c$ be a simply connected Lie group with Lie algebra 
$\g^c = \fh + i\fq$ and define 
$\cL_x$, $x \in \g$, on its maximal domain in the Hilbert subspace 
$\cH_{D_+} \subeq C^{-\infty}(M_+)$. 
Then there exists a unique smooth unitary representation $(\pi^c,\cH_{D_+})$ 
of $G^c$ 
such that 
\begin{description}
\item[\rm(i)] $\oline{\dd\pi^c}(x) = \cL_x$ for $x \in \fh$. 
\item[\rm(ii)] $\oline{\dd\pi^c}(iy) = i\cL_y$ for $y \in \fq$. 
\end{description}
\end{theorem}

\begin{ex} Reflection positive representations for $\Mot(\R^d)$ 
lead in particular to reflection positive representations 
of the compact symmetric Lie group 
$(G,H,\tau) = (\OO_d(\R), \OO_{d-1}(\R), \tau)$ which then leads to 
unitary representations of the $c$-dual group $G^c = \SO_{1,d-1}(\R)_0 
= L_+^\uparrow$. In particular, there are non-trivial reflection positive 
representations for compact groups. Here the case 
$d= 2$ and the free spin zero fields of mass $m > 0$ are of particular interest 
(cf.\ \cite{NO13}). 
\end{ex}

\begin{ex} We now 
consider $M = \R^d$, $M_+ = \R^d_+$, 
$\tau(x_0, \bx) = (-x_0, \bx)$ and the euclidean motion group 
$G = \R^d \rtimes \OO_d(\R)$. 
Then the  $G$-invariance of a distribution $D^\sharp$ on $M \times M$ means that 
it is determined by an $\OO_d(\R)$-invariant distribution $D \in C^{-\infty}(M)$ by 
\[ D^\sharp(\phi \otimes \psi) := D(\check\phi * \psi).\] 
For any reflection positive rotation invariant distribution $D \in C^{-\infty}(\R^d)$, 
we thus obtain a reflection positive representation 
$(\pi_\cE, \cE)$ of $G$ and a representation 
of the group $G^c = \R^d \rtimes \Spin_{1,d-1}(\R)$ on $\hat\cE \cong \cH_{D_+}$. 

The inclusion $\SO_{d-1}(\R) \to \OO_{1,d-1}(\R), g \mapsto \id_\R \times g$ 
induces a surjective homomorphism $\pi_1(\SO_{d-1}(\R)) \to \pi_1(\OO_{1,d-1}(\R))$, 
and since $\pi^c$ is compatible with the unitary representation 
$\hat\pi_H$ of $H$ on $\hat\cE$, it follows that $\pi^c$ factors through a 
representation of the connected Poincar\'e group $\R^d \rtimes \SO_{1,d-1}(\R)$.
\end{ex}

\appendix 

\section{Positive definite kernels and functions} 
\mlabel{app:d}

In this appendix we collect some definitions and results concerning 
positive definite functions and kernels. 

\begin{defn} \mlabel{def:1.5} Let $X$ be a set and $\cF$ be a complex Hilbert space.

\par (a)  A function $K \: X \times X \to B(\cF)$ is called a 
{\it $B(\cF)$-valued kernel}. A $B(\cF)$-valued kernel $K$ on $X$ is said to be 
{\it positive definite} if, 
for every finite sequence $(x_1, v_1), \ldots, (x_n,v_n)$ in $X \times \cF$, 
\[ \sum_{j,k = 1}^n \la K(x_j, x_k)v_k, v_j \ra \geq 0. \] 

\par (b) If $(S,*)$ is an involutive semigroup, then a 
function $\phi \: S \to B(\cF)$ is called {\it positive definite} 
if the kernel $K_\phi(s,t) := \phi(st^*)$ is positive definite. 
\end{defn}

Positive definite kernels can be characterized as those 
for which there exists a Hilbert space $\cH$ and a 
function $\gamma \: X \to B(\cH,\cF)$ such that 
\begin{equation}
  \label{eq:evalprod}
K(x,y) = \gamma(x)\gamma(y)^* \quad \mbox{ for } \quad x,y \in X 
\end{equation}
(cf.\ \cite[Thm.~I.1.4]{Ne00}). 
Here one may assume that the vectors 
$\gamma(x)^*v$, $x \in X, v \in \cF$, span a dense subspace of 
$\cH$. If this is the case, then the pair $(\gamma,\cH)$ is called a {\it realization 
of $K$}. 
The map $\Phi \: \cH \to \cF^X, \Phi(v)(x) := \gamma(x)v$, 
then realizes $\cH$ as a Hilbert subspace of $\cF^X$ 
with continuous point evaluations $\ev_x \: \cH \to \cF, f \mapsto f(x)$. 
Then $\Phi(\cH)$ is the unique Hilbert space in $\cF^X$ with continuous point evaluations 
$\ev_x$, for which $K(x,y) = \ev_x \ev_y^*$ for $x,y \in X$. 
We write $\cH_K \subeq \cF^X$ for this subspace and call it 
the {\it reproducing kernel Hilbert space with kernel~$K$}.
The dense subspace spanned by the elements of the form 
$K_m^* v$, $v \in \cF, m \in M$, is denoted $\cH_K^0$. 

For $\cF = \C$, we also write $K_m$ for the function (which corresponds to 
$K_m^*1$) which represents the evaluation in $m$ in the sense that 
$f(m) = \la f, K_m\ra$ for $f \in \cH_K$. 

\begin{ex} \mlabel{ex:vv-gns} (Vector-valued GNS construction) 
(cf.\ \cite[Sect.~3.1]{Ne00}) Let $(\pi, \cH)$ be a representation of the 
unital involutive semigroup $(S,*)$, $\cF \subeq \cH$ be a closed subspace for which 
$\pi(S)\cF$ is total in $\cH$ and $P \: \cH \to \cF$ denote the orthogonal projection. 
Then $\phi(s) := P\pi(s)P^*$ is a $B(\cF)$-valued positive definite function 
on $S$  with $\phi(\1) = \1_\cF$ because $\gamma(s) := P\pi(s) \in B(\cH,\cF)$ 
satisfies 
\[ \gamma(s)\gamma(t)^* = P\pi(st^*)P^* = \phi(st^*).\] 
The map 
\[ \Phi \: \cH \to \cF^S, \quad \Phi(v)(s) = \gamma(s)v = P \pi(s)v \] 
is an $S$-equivariant realization of $\cH$ as the reproducing kernel space
$\cH_\phi \subeq \cF^S$, on which $S$ acts by right translation, i.e., 
$(\pi_\phi(s)f)(t) = f(ts)$. 

Conversely, let $S$ be a unital involutive semigroup 
and $\phi \: S \to B(\cF)$ be a positive definite function with 
$\phi(\1) = \1_\cF$. 
Write $\cH_\phi \subeq \cF^S$ for the corresponding reproducing kernel space and 
$\cH_\phi^0$ for the dense subspace spanned by 
$\ev_s^*v, s \in S, v \in \cF$. 
Then $(\pi_\phi(s)f)(t) := f(ts)$ defines a 
$*$-representation of $S$ on $\cH_\phi^0$. 
We say that $\phi$ is {\it exponentially bounded} if 
all  operators $\pi_\phi(s)$ are bounded, so that we actually 
obtain a representation of $S$ by bounded operators on $\cH_\phi$. 
As $\1_\cF = \phi(\1) = \ev_\1 \ev_\1^*$, the map 
$\ev_\1^* \: \cF \to \cH$ is an isometric inclusion, so that we may identify 
$\cF$ with a subspace of $\cH$. Then $\ev_\1 \: \cH \to \cF$ corresponds to the 
orthogonal projection onto $\cF$ and 
$\ev_\1 \circ \pi_\phi(s) = \ev_s$ leads to 
\begin{equation}
  \label{eq:factori}
\phi(s) = \ev_s \ev_\1^* =\ev_\1 \pi_\phi(s) \ev_\1^*.
\end{equation}

If $S = G$ is a group with $s^* = s^{-1}$, then $\phi$ is always exponentially bounded and the 
representation $(\pi_\phi, \cH_\phi)$ is unitary. 
\end{ex}

\end{document}